\documentclass[11pt]{amsart}

\usepackage{hyperref} 
\usepackage[dvips]{graphicx}
\usepackage[T1]{fontenc}
\usepackage{mathptmx}
\usepackage{mathtools,slashed}
\usepackage{color}
\usepackage{bm}

\DeclareMathOperator{\Ad}{Ad}

\DeclareMathOperator{\di}{diag}
\DeclareMathOperator{\odi}{off-diag}

\newcommand{\bbar}{\begin{pmatrix}}
\newcommand{\ebar}{\end{pmatrix}}

\newcommand{\bdm}{\begin{displaymath}}
\newcommand{\edm}{\end{displaymath}}
\newcommand{\beq}{\begin{equation}}
\newcommand{\beqa}{\begin{eqnarray}}
\newcommand{\beqas}{\begin{eqnarray*}}
\newcommand{\eeq}{\end{equation}}
\newcommand{\eeqa}{\end{eqnarray}}
\newcommand{\eeqas}{\end{eqnarray*}}
\newcommand{\dd}{\textup{d}}

\newcommand{\E}{{\mathbb E}}
\newcommand{\HH}{{\mathbb H}}
\newcommand{\LL}{{\mathbb L}}
\newcommand{\SSS}{{\mathbb S}}
\newcommand{\he}{\mathcal{H}}
\newcommand{\lhe}{\mathfrak{h}}
\newcommand{\R}{{\mathbb R}}
\newcommand{\C}{{\mathbb C}}
\newcommand{\D}{{\mathbb D}}

\newcommand{\real}{{\mathbb R}}

\newcommand{\Z}{{\mathbb Z}}

\newcommand{\Nil}{\hbox{Nil}^3}
\newcommand{\Nili}{{\hbox{Nil}^3_1}}

\newcommand{\nili}{{\mathfrak{nil}^3_1}}
\newcommand{\SU}{SU(2)}
\newcommand{\su}{{\mathfrak {su}}(2)}

\newcommand{\LSU}{\Lambda SU(2)_\sigma}
\newcommand{\LSL}{\Lambda SL(2,  \mathbb C)_{\sigma} }
\newcommand{\LSLP}{\Lambda^+ SL(2, \mathbb C)_{\sigma} }
\newcommand{\LSLN}{\Lambda^-_* SL(2, \mathbb C)_{\sigma} }
\newcommand{\Up}{U_{\mathfrak{p}}}
\newcommand{\Uk}{U_{\mathfrak{k}}}

   \newtheorem{theorem}{Theorem}[section]
   
   \newtheorem{corollary}[theorem]{Corollary}
   \newtheorem{lemma}[theorem]{Lemma}
   \newtheorem{definition}[theorem]{Definition}

 \theoremstyle{remark}
   \newtheorem{example}[theorem]{Example}
   \newtheorem{remark}[theorem]{Remark}
   
\numberwithin{equation}{section}

\begin{document}

\title[Maximal surfaces in the Heisenberg group]{Maximal surfaces in the Lorentzian Heisenberg group}

\begin{abstract}
The 3-dimensional Heisenberg group can be equipped with three different types of  left-invariant Lorentzian metric, according to whether the center of the Lie algebra is spacelike, timelike or null. Using the second of these types, we study
spacelike surfaces of mean curvature zero.  These surfaces with singularities are associated with
harmonic maps into the 2-sphere.  We show that the generic singularities are cuspidal edge, swallowtail and cuspidal cross-cap. We also give the loop group construction for these surfaces, and the criteria on the loop group potentials for the different generic singularities. Lastly, we solve the Cauchy problem for
harmonic maps into the 2-sphere using loop groups,
and use this to give a geometric characterization of the singularities. We use these results to prove
that a regular spacelike maximal disc with null boundary must have at least
two cuspidal cross-cap singularities on the boundary.
 \end{abstract}

\author{David Brander}
\address{Department of Applied Mathematics and Computer Science, Technical University of Denmark \\
Richard Petersens Plads, Building 324,  DK-2800, Kgs. Lyngby,  Denmark}
\email{dbra@dtu.dk}

\author{Shimpei Kobayashi}
\address{Department of Mathematics, Hokkaido University\\
Nishi 8-Chome  Kita 10-Jou Kita-Ku\\
Sapporo 060-0810, Japan}
\email{shimpei@math.sci.hokudai.ac.jp}
\thanks{The first named author is partially suported by Independent Research Fund Denmark, grant 9040-00196B.
The second named author is partially supported by Kakenhi 22K03265.}

\keywords{Differential geometry, harmonic maps, loop groups, maximal surfaces, Heisenberg group, singularities}
\subjclass[2020]{Primary~53A10; Secondary~53C42,~53C43}

%\date{\today}
%\thanks{}

\maketitle

\section{Introduction}
\subsection{Background}
Surfaces of mean curvature zero in the Heisenberg group $\he$, equipped with
a left-invariant Riemannian metric, denoted here $\Nil$, have  been studied by many authors. Berdinskii and Taimanov \cite{BerTai2005} gave a spinor type representation, similar to other so-called integrable classes of surfaces such as constant mean curvature surfaces in space forms.  One can show (\cite{figueroa1999,figueroa2007,daniel2011}) that
they are related to harmonic maps into the hyperbolic space $\HH^2$ as follows: the Gauss map of an immersion $f$ into $\Nil$ is the left translation $N=f^{-1}n$  of
the unit normal. The surface is called \emph{vertical} at a point $p$ if $N_p$ is parallel to the $e_1e_2$-plane. The Gauss map of a nowhere vertical minimal surface in $\Nil$ is harmonic into the hemisphere equipped with the hyperbolic metric, and conversely a nowhere holomorphic harmonic map into $\HH^2$ is the Gauss map of a minimal surface.
This allows one to use the theory of harmonic maps to study these surfaces (see, e.g., \cite{cartier2011, dik2016}).

If one considers now Lorentzian metrics on $\he$, Rahmani \cite{rahmani1992} showed that there are  (up to homothety) three different choices of left-invariant metric (see Section \ref{metricsect} below), 
$g_1$, $g_2$ and $g_3$, according to whether the $1$-dimensional center of the Lie algebra is spacelike, timelike or null in the metric. A classical Weierstrass-type representation exists
for mean curvature zero surfaces in such spaces (\cite{lmm2011,cmo2017}), but the Weierstrass data need to satisfy non-trivial extra conditions; hence
only simple examples of mean curvature zero surfaces in
these manifolds have been given. 
For spacelike surfaces in the case of the metric $g_2$ (see \cite{hlee2011} and the present article), and for timelike surfaces with the metric $g_1$ (see \cite{kiko2022}), a  relationship between harmonic maps and mean curvature zero surfaces, similar to the Riemannian case exists. This allows one to use loop group methods to construct all solutions, and, in particular, generate many more examples.

Here we will study spacelike mean curvature zero surfaces in $\he$ equipped with the metric $g_2$, denoting this space by $\Nili$, via the relationship with harmonic maps into $\SSS^2$. 

\subsection{Surfaces with harmonic Gauss maps} \label{harmonicintro}
Harmonic maps from a Riemann surface into $\SSS^2$ are directly related to various different geometric problems, via the Gauss map. 
In these cases, the passage from the harmonic Gauss map to the surface is well defined even at points where the resulting surface is not regular (a so-called \emph{frontal}), so it is natural in this context to consider generalized surfaces with singularities.

Given a harmonic map $\nu: M \to \SSS^2$, where $M$ is a simply connected Riemann surface, there are naturally associated:
\begin{enumerate} 
 \item A geometrically unique constant positive Gaussian curvature $K=1$, surface $f_{cgc} : \to \real^3$ (see, e.g. \cite{spherical}), satisfying a compatible pair of linear PDE:
 \[
  f_z = i \nu \times \nu_z,
 \]
that has $\nu$ as its Gauss map. Here $z=x+iy$ and $f_z:= (1/2)(\partial f/\partial x-i\partial f/\partial y)$.
 \item A pair of geometrically unique constant mean curvature $H=1/2$ surfaces $f^{\pm}: M \to \real^3$ that are parallel surfaces to $f_{cgc}$, given by:
 \[
  f^\pm = f_{cgc} \pm \nu.
  \]
 \item A 2-dimensional family of zero mean curvature spacelike surfaces $f$ in $\Nili$, to be described below.
\end{enumerate}
Conversely, such surfaces have  harmonic Gauss maps. 

The way that singularities appear in each case is different however: evidently
$f_{cgc}$ is regular if and only if the harmonic map $N$ is also regular, by
the relation $f_z = i N \times N_z$. 

For the second case, the CMC surface $f^\pm$ has derivatives:
\beq \label{eq:cmc}
 f^\pm_x = \nu \times \nu_y \,  \pm \nu_x, 
 \quad \quad
 f^\pm_y = -\nu \times \nu_x \,  \pm \nu_y.
\eeq
It follows that $f^\pm$ only has singularities of rank zero (branch points), and these occur  precisely when
\[
 \nu_x = \mp \nu \times \nu_y.
\]
If $\nu$ itself has rank zero at a point (i.e. $\nu_x=\nu_y=0$), then both $f^+$ and $f^-$ 
have a branch point. If $\nu_x = \mp \nu \times \nu_y \neq 0$, then
$f^\pm$ has a branch point, and, one can show, $f^\mp$ has an umbilic point.

For the third case, the fact that the metric on $\Nili$ is not isotropic  means that a given harmonic map $\nu$ into $\SSS^2$ can be interpreted in many geometrically distinct ways as the Gauss map  for a surface in $\Nili$. The Lie algebra $\nili$ is identified with the vector space $\real^3$, with the center of $\nili$ spanned by $e_3$.
If we fix the given embedding of $\SSS^2$ in $\real^3$, with the north pole pointing in the $e_3$ direction, then we obtain a well defined 
mean curvature zero surface $f: M \to \Nili$ (described below).
If we apply an isometry to $\SSS^2$, then this has no geometric 
significance to $\nu$.  However, a rotation about any axis other than the  $e_3$-axis is not an isometry of $\Nili$, and so there is a 2-dimensional family of
surfaces $f: M \to \Nili$ associated to the harmonic map $\nu$.
As for the singularities, $f$ will have singularities exactly at points 
where $\nu$ is perpendicular to $e_3$, i.e. when $\nu$ takes values 
in the equator in the $e_1 e_2$-plane. Hence the singular set is different for each surface in the $2$-parameter family (see Figure \ref{figex0}, and Examples \ref{example0} and \ref{example1}).

  \begin{figure}[htb]
\centering
$
\begin{array}{ccc}
\includegraphics[height=30mm]{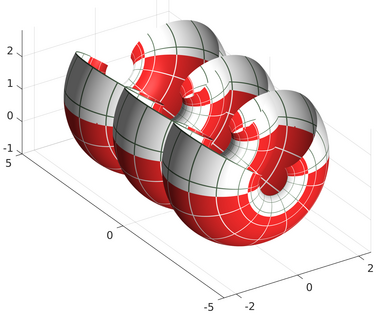}  
 & \quad
\includegraphics[height=30mm]{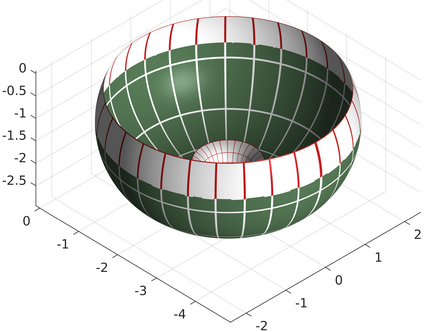} \quad & 
\includegraphics[height=30mm]{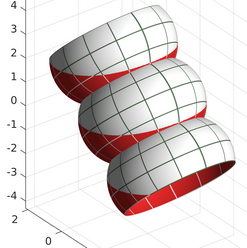} 
\vspace{2ex} \\
\includegraphics[height=35mm]{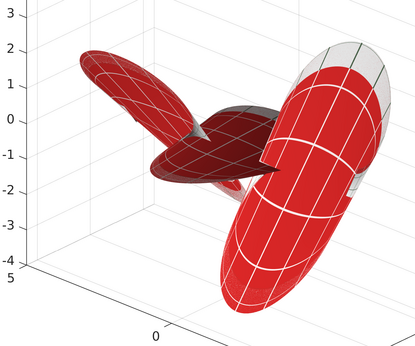}  \quad & 
\includegraphics[height=35mm]{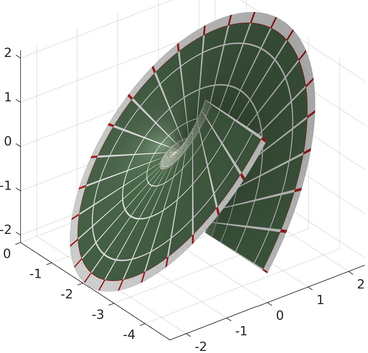}   \quad & 
\includegraphics[height=35mm]{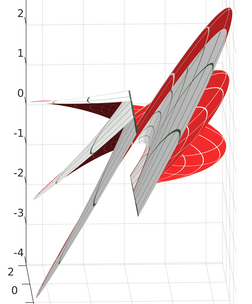} 
\end{array}
$
\caption{Top: Portions of a single constant mean curvature surface of revolution, with three different spatial orientations. Bottom: the corresponding maximal surfaces in $\Nili$, in the same order.
(Example \ref{example0}.)}
\label{figex0}
\end{figure}

\subsection{Surfaces with singularities}
It is well known that there are no complete maximal surface in the Lorentzian 3-space $\LL^3$ besides the plane. Nevertheless, such surfaces are naturally of interest to geometers,
and this motivated the definition of a \emph{maxface} \cite{uy2006}
as a generalized surface with singularities.  For maxfaces, the Gauss map is a holomorphic map into the extended hyperbolic plane, and the surface has singularities where the Gauss map takes values in the boundary between the two components. The generic singularities  are cuspidal edge, swallowtail and cuspidal cross-cap \cite{fsuy}.

In contrast to this, for maximal surfaces in $\Nili$,  the Gauss map is a nowhere holomorphic harmonic map into $\SSS^2$. However, we will show that the generic singularities are of the same type as for maxfaces.

\subsection{Outline of this article}
In Section \ref{metricsect} we briefly describe the Heisenberg group and the left-invariant Lorentzian metric we will use.
In Section \ref{maxsection} we describe the spinor representation for maximal surfaces in $\Nili$, the relation with harmonic maps into $\SSS^2$, and how to produce all solutions via loop group methods.

In Section \ref{singsection} we prove that the generic singularities 
are cuspidal edge, swallowtail and cuspidal cross-cap: Theorem \ref{thm:criteria} and Theorem \ref{thm:criteria2} characterize these singularities respectively in terms of the conformal Gauss map and the
Abresch-Rosenberg differential.  This gives (in Theorem \ref{thm:generic})
a characterization of
the generic singularities in the space of local solutions.

In Section \ref{sec:cauchy} we show how to solve the Cauchy
problem for a harmonic map into $\SSS^2$ (Theorem \ref{cauchythm}) via loop group methods, which is of independent interest.
This allows one to construct harmonic maps with prescribed transverse derivative along any curve in the sphere: in particular, the equator, and thus with a prescribed singular set along a curve (Theorem \ref{singcauchythm}). Theorem \ref{singcauchythm2} shows explicitly the relation between the geometry of the Cauchy data and the type of generic singularity at a point.

In Section \ref{sect:equatorial} we look at the geometry of
the curve on the associated CMC surface that corresponds
to the equator in $\SSS^2$ for the harmonic Gauss map.
This curve gives a very simple characterization of the generic singularities of the maximal surface in $\Nili$.
As an application, we show in Theorem \ref{thm:twocrosscap}
that a regular spacelike maximal disc with null boundary
in $\Nili$ must have at least two cuspidal cross-caps on the boundary.

\textbf{Numerics:} The DPW method for harmonic maps can be implemented numerically to compute solutions from a given potential. 
At the time of writing, implementations in Matlab that allow one to compute the examples shown here, as well as any other solution with given Cauchy data, can be found at \href{http://davidbrander.org/software.html}{http://davidbrander.org/software.html}

\section{Left invariant metrics on the Heisenberg group} \label{metricsect}
The 3-dimensional Heisenberg group $\he$ is the group of $3\times 3$ real upper-triangular matrices with one's on the diagonal.
The Lie algebra $\lhe$ is spanned by:
\[
 a_1 = \bbar 0 & 1 & 0\\0&0&0\\0&0&0\ebar,
 \quad
 a_2 = \bbar 0 & 0 & 0\\0 & 0 & 1\\0&0&0\ebar,
 \quad
 a_3 = \bbar 0 & 0 & 1\\0&0&0\\0&0&0\ebar,
\]
with commutators, $[a_1,a_2]=a_3$ and
$[a_1,a_3]=[a_2,a_3]=0$.
For this group, exponential coordinates give a bijection
$\exp: \lhe \to \he$, with the formula:
\[
 \exp(x_1 a_1 + x_2 a_2 +x_3a_3) = \bbar 1& x_1& x_3+\frac{1}{2}x_1x_2\\
         0 & 1 & x_2\\ 0 & 0 & 1 \ebar.
\]
Thus, identifying $\he$ with $\lhe = \real^3$, it is common to
represent $\he$ as $\real^3$ with the group structure that corresponds to the matrix product in $\he$, namely:
\[
 (x_1,x_2,x_3)\cdot(\tilde x_1,\tilde x_2, \tilde x_3) = 
 (x_1+\tilde x_1, \,\, x_2+\tilde x_2, \,\, x_3+\tilde x_3 + \tfrac{1}{2}(x_1\tilde x_2 - \tilde x_1 x_2)).
\]
Differentiating the map $L_X: \real^3 \to \real^3$, given by $L_X(\tilde X)=X\cdot \tilde X$, at $\tilde X= e = (0,0,0)$, we have:
\[
 \dd L_X |_e = \bbar 1 & 0 & 0 \\ 0 & 1 & 0 \\ -\tfrac{1}{2}x_2 & \tfrac{1}{2} x_1 & 1 \ebar,
\]
so the left invariant vector fields generated by $a_1$, $a_2$ and $a_3$ are:
\beq \label{livfs}
 A_1 = \frac{\partial}{\partial x_1} - \frac{x_2}{2}\frac{\partial }{\partial x_3} ,\quad \quad
 A_2 = \frac{\partial}{\partial x_2} + \frac{x_1}{2}\frac{\partial }{\partial x_3} ,\quad \quad
 A_3 = \frac{\partial }{\partial x_3}.
\eeq

\subsection{Left invariant Riemannian metrics}
Milnor \cite{milnor1976} introduced a systematic study of left invariant Riemannian metrics on Lie groups. On a connected Lie group of dimension $3$, the Lie bracket is necessarily given by $[u,v]=L(u \times v)$, where $L$ is linear and the cross-product is defined with respect to a chosen orientation. The group is unimodular if and only if the linear map $L$ is self-adjoint. Using this, Milnor
showed that, for the case of a 3-dimensional unimodular Lie group $G$, a basis $\{e_1, e_2, e_3\}$ can be chosen for the Lie algebra that is orthonormal with respect to the given metric and for which the Lie bracket is given by
\[
[e_1, e_2]=\lambda_3 e_3, \quad
[e_2,e_3]=\lambda_1 e_1, \quad
[e_3, e_1]=\lambda_2 e_2.
\]
For the case of the Heisenberg group $\he$, with notation as above, since all commutators are proportional to $a_3$, it follows that one of the $e_j$, which we may as well take to be $e_3$, is proportional to $a_3$. Then  $\lambda_1=\lambda_2=0$, $\lambda_3 \neq 0$, and no generality
is lost by assuming that
\[
e_1=a_1, \quad e_2 = a_2, \quad e_3 = \frac{1}{\lambda_3} a_3.
\]
Since the orientation can be chosen so that $\lambda_3>0$,  there is, up to a positive homothety, only one left invariant metric on $\he$.

\subsection{Left invariant Lorentzian metrics}
Rahmani \cite{rahmani1992} studied the \emph{Lorentzian} metric version of this problem for 3-dimensional unimodular Lie groups, using a similar approach to Milnor's. Since a Lorentzian metric is not isotropic, there are more possibilities. One can again choose 
an orthonormal basis for the Lie algebra, with
\[
\langle e_1, e_1 \rangle = \langle e_2, e_2 \rangle =1,
\quad \langle e_3, e_3 \rangle =-1,
\]
and the Lie bracket is given by 
$[u, v]=L (u \times v)$, where $L$ is linear and the Lorentzian 
cross-product is defined by $e_1 \times e_2 = - e_3$,  $e_2 \times e_3 = e_1$, $e_3 \times e_1 = e_2$. Here, also, the group is unimodular if and only if $L$ is self-adjoint. The different possible Lie algebra structures in terms of the orthonormal basis are given in \cite{rahmani1992}; among these, there are three that give the Heisenberg group structure:
\begin{enumerate}
 \item  $[e_1,e_2]=[e_1,e_3]=0$,  \quad $[e_2,e_3]=\mu e_1$.
 \item $[e_2,e_3]=[e_1,e_3]=0$,  \quad $[e_1,e_2]= \mu e_3$.
 \item $[e_2,e_3]=0$, \quad $[e_3, e_1]=[e_2,e_1]=e_2-e_3$.
\end{enumerate}
These correspond  to the center $\hbox{Span} \{a_3\}$ of $\lhe$  being spacelike, timelike, or null: for the three cases we respectively have
$a_3= \mu e_1$, $a_3= \mu e_3$ and $a_3=e_3-e_2$.
The corresponding left invariant Lorentzian metrics are denoted respectively by $g_1$, $g_2$ and $g_3$. 
The algebra of Killing fields is given for each case in \cite{rahmani1992}, showing that the isometry groups are of dimension $4$ for $g_1$ and $g_2$, and of dimension $6$ for the case of $g_3$. The geometry of the metrics $g_1$, $g_2$ and $g_3$ is studied further by N. Rahmani and S. Rahmani in \cite{rahmani2006}.  They prove that the metrics are nonisometric, and that $g_3$ is flat.

\section{Maximal spacelike surfaces in \texorpdfstring{$\Nili$}{Nil3}}	 \label{maxsection}
In this article, we use the metric $g_2$, in particular with $\mu=1$,  taking 
\[
e_1=a_1, \quad  e_2=a_2, \quad  e_3=a_3,
\]
as the orthonormal basis for the Lorentzian metric. In the coordinates $(x_1,x_2,x_3)$ for $\real^3=\lhe\equiv\he$,
we have from \eqref{livfs}:
\[
 \frac{\partial }{\partial x_1} = e_1 + \frac{x_2}{2}e_3, \quad \quad
 \frac{\partial }{\partial x_2} = e_2 - \frac{x_1}{2}e_3,
 \quad \quad \frac{\partial }{\partial x_3}=e_3,
\]
in terms of the left invariant vector fields $e_i= a_i$. 
Using the Lorentzian orthonormality, with $\langle e_3,e_3\rangle = -1$, this gives the metric in the coordinate frame:
\[
g_2  = \dd x_1^2+\dd x_2^2 -\left(\dd x_3 + \frac{1}{2}(x_2 \dd x_1 - x_1 \dd x_2)\right)^2.
\]
   The Levi-Civita connection $\nabla$ of $g_2$ is given by 
\[
 \begin{pmatrix}
 \nabla_{e_1} e_1 & \nabla_{e_1} e_2  & \nabla_{e_1} e_3 \\
 \nabla_{e_2} e_1 & \nabla_{e_2} e_2  & \nabla_{e_2} e_3 \\
 \nabla_{e_3} e_1 & \nabla_{e_3} e_2  & \nabla_{e_3} e_3 
 \end{pmatrix}
 =
 \begin{pmatrix}
 0 & \tfrac{1}{2} e_3& \tfrac{1}{2} e_2\\
-\tfrac{1}{2} e_3  & 0 & - \tfrac{1}{2} e_1\\
\tfrac{1}{2} e_2    & -\tfrac{1}{2} e_1& 0
 \end{pmatrix}.
\]
 Then it is easy to see the symmetric bi-linear map given by $\{e_i, e_j\} = \nabla_{e_i} e_j + \nabla_{e_j} e_i$ satisfies
 \[
     \{e_1, e_2\}=0, \quad \{e_1, e_3\}=  e_2 \quad \mbox{and}\quad \{e_2, e_3\}=- e_1.
 \]
 
 \subsection{Spinor representation for spacelike conformal immersions in \texorpdfstring{$\Nili$}{Nil3}}
 We first give a representation of an arbitrary conformal immersion into $\Nili$ in terms of a pair of \emph{generating spinors}, $\psi_1$, $\psi_2$.  After that, we characterize the mean curvature zero property in terms of the spinors. An $SU(2)$-valued frame can be used to construct both the spinors and the harmonic Gauss map.

Consider a conformal spacelike immersion $f$ from a Riemann surface $M$ into 
 $\Nili$. Choose a conformal coordinate $z=x+i y \in \D \subset M$ in a simply 
 connected domain $\D$ in $M$ and define complex valued functions $\phi_1, \phi_2$ and $\phi_3$ by 
 \[
     f^{-1} f_z = \phi_1 e_1 + \phi_2 e_2 + \phi_3 e_3,
 \]
 where the subscript denotes the derivative with respect to $z$, that is, $\partial_z = \tfrac12
 (\partial_x - i \partial_y)$.
 Setting $\varPhi = f^{-1}f_z$ and $\overline \varPhi = \overline{f^{-1}f_z} = f^{-1}f_{\bar z}$ and $\beta = f^{-1} df = \varPhi dz + \overline \varPhi d \bar z$, the Maurer-Cartan equation $d \beta + \frac{1}{2}[\beta \wedge \beta]=0$ 
  and  the second fundamental form of $f$ can be rephrased as, 
  see \cite[Section 1.1]{dik2016}
\begin{equation}\label{eq:fundamental}
 \varPhi_{\bar z} - \overline{\varPhi}_z + [\overline{\varPhi},\varPhi]=0 
 \quad\mbox{and}\quad
 \varPhi_{\bar z} + \overline{\varPhi}_z + \{\overline{\varPhi},\varPhi\}=e^{u} f^{-1} \boldsymbol{H}, 
\end{equation}
 where $e^u dz d \bar z$ and $\boldsymbol{H}$ denote
 the induced conformal metric and the mean curvature vector, respectively.
 Note that $\partial_{\bar z} = \tfrac12( \partial_x + i \partial_y)$, and 
 $u : \D \to \R$ is a real-valued function on $\D$.
 
 We now rephrase the fundamental equations in \eqref{eq:fundamental} 
 in terms of generating spinors. Since $f$ is a conformal immersion, it is easy to see that
\begin{equation}\label{eq:conditionphi}
\phi_1^2+ \phi_2^2 - \phi_3^2 =0\quad\mbox{and}\quad |\phi_1|^2 +|\phi_2|^2-|\phi_3|^2  = \frac{1}{2} e^u.
\end{equation}
 The \textit{generating spinors} $\psi_1$ and $\psi_2$ are defined 
  as a solution of $\phi_1^2+ \phi_2^2 - \phi_3^2 =0$ with
 \[
 \phi_1 = (\overline{\psi_2})^2 - \psi_1^2, \quad 
 \phi_2 = i ((\overline{\psi_2})^2 + \psi_1^2), \quad
 \phi_3 = 2 i \psi_1 \overline{\psi_2},
 \]
 where $\overline{\psi_2}$ denotes the complex conjugate of $\psi_2$. Note that 
 $\psi_1 \sqrt{d z}$ and  $\overline{\psi_2} \sqrt{d z}$ are well defined on $M$.

 The conformal factor $e^{u}$ of the  induced metric $\langle df,df\rangle$ 
 can be expressed  by the spinors  $\psi_1, \psi_2$ via the second formula at \eqref{eq:conditionphi}:
 \[
 e^{u} =4 (|\psi_1|^2-|\psi_2|^2)^2.
 \]
 In the following, we assume the regularity condition $|\psi_1|\neq|\psi_2|$.
  A straightforward computation shows that 
 \begin{equation}\label{eq:cross}
     f^{-1} f_z  \times f^{-1} f_{\bar z} 
     = i e^{u/2}
   \left( 2 \Im (\psi_1 \psi_2 )e_1- 2 \Re (\psi_1 \psi_2)e_2- (|\psi_1|^2 + |\psi_2|^2)e_3\right).
 \end{equation}
  Let us denote the unit normal vector field of $f$ by $N$.  
  From \eqref{eq:cross},
  $N$ is given by 
  \begin{equation}\label{eq:unitnormal}
  N=e^{-u/2}L, 
  \quad \quad
  f^{-1} L = 
  2 \left( 2 \Im (\psi_1 \psi_2 )e_1- 2 \Re (\psi_1 \psi_2)e_2-
   (|\psi_1|^2 + |\psi_2|^2)e_3\right).
  \end{equation}
The \textit{support function} $h$ of $f$, with respect to $z$, (which is evidently non-vanishing for a spacelike surface) is defined as: 
 \begin{equation}\label{eq:support}
 h :=  \langle f^{-1} L, e_3 \rangle = 2(|\psi_1|^2+|\psi_2|^2).
 \end{equation}

  It is straightforward to check the following theorem, see \cite[Theorem 3.5]{dik2016}
   for the case of a surface in $\Nil$:
 \begin{theorem}\label{thm:Dirac}
 The pair of generating spinors $\{\psi_1, \psi_2\}$ satisfies the 
  following \textit{nonlinear Dirac equation}, 
 \begin{equation}\label{Dirac1}
\slashed{D} \begin{pmatrix} 
\psi_1

\\ 
\psi_2
\end{pmatrix} 
:=
\begin{pmatrix}
\partial_{z}\psi_{2}+\mathcal{U}\psi_1
\\
-\partial_{\bar z}\psi_1+\mathcal{V}\psi_2
\end{pmatrix} 
=
\left(
\begin{array}{c}
0
\\
0
\end{array}
\right),
\end{equation}
 where 
 \[
 \mathcal U = \mathcal V = \frac{H}{2}e^{u/2}i - \frac{1}{4}h.  
 \]
 Here $H$, $e^{u}$ and $h$ are  
 the mean curvature, the conformal factor 
 and the support function for $f$
 respectively. 
 
 Moreover, together with the second fundamental form, 
  the vector $\tilde \psi = (\psi_1, \psi_2)$ satisfies the following system
\begin{equation}\label{eq:Weingarten}
\tilde \psi_z = \tilde \psi \tilde U, \quad
\tilde \psi_{\bar z}= \tilde \psi \tilde V,
\end{equation}
where 
\begin{equation}\label{eq:tildeUV}
    \tilde U = \begin{pmatrix}
    \frac12 w_z - \frac{i}2 H_z e^{-w/2+u/2}&- e^{w/2} \\ 
     B e^{-w/2}& 0
    \end{pmatrix}, \quad 
    \tilde V = \begin{pmatrix}
        0 & - \bar B e^{-w/2} \\ e^{w/2}&   \frac12 w_{\bar z} - \frac{i}2 H_{\bar z} e^{-w/2+u/2}
    \end{pmatrix}.
\end{equation}
 Here the function $e^{w/2}$ is the {\rm Dirac potential} defined by
\[
 e^{w/2} = \mathcal U = \mathcal V =\frac{H}{2}e^{u/2}i - \frac{1}{4}h,
\]
 and the quadratic differential $B \, dz^2$ will be called 
the {\rm Abresch-Rosenberg differential} given by 
 \[
  B = \frac{2 i H + 1}{2}\left( \psi_1 \overline{\psi_2}_z - \overline{\psi_2} {\psi_1}_z \right) + 2i H (\psi_1\overline{\psi_2} )^2.
 \]
 \end{theorem}

 The unit normal $f^{-1} N$ can be considered as a map into the 
 the union of two hyperbolic two-spaces
 $\mathbb H^2_+ \cup \mathbb H^2_- \subset \mathbb E^3_1 (= \mathfrak{nil}^3_1)$,
 where $\mathbb(b H^2_{+}$ (resp. $\mathbb H^2_-$) denotes the  hyperbolic two-space with positive (resp. negative)
 $e_3$-component. 
 Note that when  $|\psi_1|>|\psi_2|$ (resp. $|\psi_2|>|\psi_1|$), 
  $f^{-1}N$ takes values in $\mathbb H^2_-$ (resp. $\mathbb H^2_+$).
 We now consider the \textit{normal Gauss map} $g$ of the surface $f$
 as a map defined as the composition of the stereographic 
 projection $\pi$ from the point $(0, 0,-1)$ with
 $f^{-1} N$ in \eqref{eq:unitnormal},  that is, $g = \pi \circ f^{-1} N: \D \to \C \cup \{\infty\} \setminus \mathbb S^1$, 
 (here we identify $(x, y, 0)$ with $ - y + i x$) and
 thus, we obtain
 \begin{equation}\label{eq:Normal}
 g= \frac{\psi_1}{\overline{\psi_2}} 
 \end{equation}
 and $f^{-1}N$ can be represented by the normal Gauss map $g$ as
 \begin{equation*}
     f^{-1} N = \frac1{|g|^2-1}\left(
      2 \Im (g) e_1 - 2 \Re (g) e_2 - (|g|^2+1) e_3\right).
 \end{equation*}
We now introduce a family of Maurer-Cartan forms $\{\alpha^{\lambda}\}_{\lambda \in \mathbb S^1}$:
 \begin{equation}\label{eq:alpha}
     \alpha^{\lambda} = U^{\lambda} d z + V^{\lambda} d \bar z,
 \end{equation}
  where
  \begin{equation*}
      U^{\lambda} = 
      \begin{pmatrix}
    \frac14 w_z - \frac{i}2 H_z e^{-w/2+u/2}&-\lambda^{-1} e^{w/2} \\ 
     \lambda^{-1}B e^{-w/2}& -\frac14 w_z
    \end{pmatrix}, \quad 
     V^{\lambda} = 
\begin{pmatrix}
    -\frac14 w_{\bar z} & - \lambda \bar B e^{-w/2} \\ \lambda e^{w/2}&   \frac14 w_{\bar z} - \frac{i}2 H_{\bar z} e^{-w/2+u/2}
\end{pmatrix}.
  \end{equation*}
  Note that $U^{\lambda}|_
 {\lambda =1}$ and $V^{\lambda}|_{\lambda =1}$ are obtained by 
 a gauge transformation applied to $\tilde U$ and $\tilde V$ in \eqref{eq:tildeUV}, that is $\Ad (G^{-1}) (\tilde U) +  G^{-1} G_z$ and
 $\Ad (G^{-1}) (\tilde V) +  G^{-1} G_{\bar z}$ with the 
  matrix $G = \di (e^{-w/4}, e^{w/4})$. 
  Then we characterize a spacelike maximal surface in $\Nili$ in terms of the family of connections 
   $d + \alpha^{\lambda}$ and the normal Gauss map $g$ as follows.
\begin{theorem}\label{thm:mincharact}
 Let $f : \D \to \Nili$ be a conformal spacelike immersion and $\alpha^{\lambda}$
 the $1$-form defined in \eqref{eq:alpha} and $g$ the normal Gauss map in \eqref{eq:Normal}.
 Then the following statements are mutually equivalent{\rm:}
 \begin{enumerate}
 \item $f$ is a maximal surface.
 \item $d + \alpha^{\lambda}$ is a family of flat connections on $\D \times  \SU$.
 \item The normal Gauss map $g$ for $f$ is a nowhere holomorphic harmonic 
       map into the $2$-sphere.
 \end{enumerate}
 \end{theorem}
 The proof is almost verbatim to the case of a surface in $\Nil$, see for example \cite{dik2016},
 thus we omit.

 From Theorem \ref{thm:mincharact}, there exists a family of maximal surfaces $\{f^{\lambda}\}_{\lambda \in \mathbb S^1}$ parameterized 
  by $ \lambda \in \mathbb S^1$ with a pair of generating spinors $\{\psi_1 (\lambda ), \psi_2 (\lambda )\}$ such that $\{\psi_1 (\lambda), \psi_2 (\lambda)\}|_{\lambda =1}$ are the generating spinors 
  of $f = f^{\lambda}|_{\lambda = 1}$.
 Moreover, we can define a map $F$ from $\D $ into $\SU$ 
 associated with respect to the generating spinors 
 $\psi_1$ and $\psi_2$ for a maximal surface:
 \begin{equation}\label{eq:extframin}
 F(\lambda) =\frac{1}{\sqrt{|\psi_1(\lambda)|^2+|\psi_2(\lambda)|^2}} 
\begin{pmatrix}
  \psi_1(\lambda) & \psi_2(\lambda) \\  - \overline{\psi_2(\lambda)} & \overline{\psi_1(\lambda)}
\end{pmatrix}. 
 \end{equation}
 Then $F$ will be called the \textit{extended frame} of the 
 spacelike maximal surface and the harmonic normal Gauss map $g = \psi_1/\overline{\psi_2}$.
 \begin{remark}
  Without loss of generality $F(\lambda)$ take values in $\LSU$.
 \end{remark}

 Using a logarithmic derivative of the extended frame $F$ with respect to 
 $\lambda$, we have a formula for 
 a conformal maximal surface in $\Nili$.
\begin{theorem}\label{thm:Sym}
 Let $F$ be the extended frame
 for a spacelike maximal surface. Define maps 
 $f_{cmc}$ and $N$ respectively by 
 \begin{equation}\label{eq:SymEuc}
  f_{cmc}=-i \lambda (\partial_{\lambda} F) F^{-1}  -N
  \quad  \mbox{and} \quad 
 N=  \Ad (F) E_3, \quad E_3 := \frac{1}{2}\bbar i & 0\\0 &-i\ebar.
 \end{equation}
 Moreover, define a map  $f^{\lambda}:\mathbb{D}\to \Nili$ by
 $f^{\lambda}:=\Xi_{\mathrm{nil}}\circ \hat{f^{\lambda}}$ with
\begin{equation}\label{eq:symNil}
 \hat f^{\lambda} = 
    \left.
    \left(f_{cmc}^o -\frac{i}{2} \lambda (\partial_{\lambda} f_{cmc})^d\right)
    \;\right|_{\lambda \in \mathbb{S}^1}, 
\end{equation}
 where the superscripts ``$o$'' and ``$d$'' denote the off-diagonal and 
 diagonal part, 
 respectively. Then, for each $\lambda \in \mathbb{S}^1$, 
 the map $f^{\lambda}$ is a maximal surface in $\Nili$ and 
 $ \pi \circ N$ is the normal Gauss map of $f^{\lambda}$, 
  where $\pi$ is the stereographic projection from the north pole.
 In particular, $f^{\lambda}|_{\lambda =1}$ gives 
 the original spacelike maximal surface up to a rigid motion.
\end{theorem}

\begin{remark}
\mbox{}
\begin{enumerate}
\item Identifying the Lie algebra $\su$ with the Euclidean $3$-space $\E^3$, with orthonormal basis $E_1 = \odi(-i,-i)/2$, 
$E_2 = \odi(1,-1)/2$ and $E_3 = \di(i,-i)/2$, the map $f_{cmc}$ in \eqref{eq:SymEuc} defines a constant mean curvature surface with $H_{cmc} =1/2$
 in $\E^3$, see for example \cite{bobenko1994}.
 The formula above for $f_{cmc}$ corresponds to the choice  $f^-$ in Section \ref{harmonicintro}. The Gauss map $g$ used here is holomorphic exactly when either $N_x = N \times N_y$, or $N_z =0$, i.e., when the map $f^-$ fails to be regular. So $f_{cmc}$ is regular for a nowhere holomorphic harmonic map.
 \item It is known that for a given nowhere holomorphic harmonic map $g$ from a surface $M$
  into $\mathbb S^2$, there exists a regular constant mean curvature $H_{cmc} \neq 0$ surface in 
  the Euclidean $3$-space, \cite{Kenmotsu1979}. In particular, the induced metric of the 
  surface is given by 
  \[
   ds^2 = \left(\frac2{H_{cmc}} \frac{| \partial_{\bar z} g|}{1 + |g|^2}  \right)^2|dz|^2.
  \]
 Moreover, there exists an extended frame $F$ corresponding to the nowhere holomorphic harmonic map $g$ and the CMC surface, see Theorem \ref{thm:mincharact} and 
 thus one can define a map map $f^{\lambda}$ by \eqref{eq:symNil}.
\end{enumerate}
 \end{remark}

\subsection{Generalized maximal surfaces}   \label{generalizedsection}
 
 We first give an alternative proof of a result that can be found
 in \cite{hlee2011}:
  \begin{theorem}
   There does not exist any complete spacelike maximal surface
    in $\Nili$.
  \end{theorem}
  \begin{proof}
  From the construction in Theorem \ref{thm:Sym}, for a spacelike maximal surface $f$ in $\Nili$, 
  there exists a corresponding CMC surface $f_{cmc}$ in the Euclidean $3$-space. The
  metric of the CMC surface is given by $ds_e^2 = h^2 dz d \bar z$, where $h$ is the support function 
  defined in \eqref{eq:support}.
  Then it is easy to see that 
 the induced metric $ds^2 = e^{u}d z d \bar z$ of the spacelike maximal surface and the induced metric $ds_e^2$ of the 
 corresponding CMC surface have the following relation:
\begin{equation}\label{eq:relation}
  e^{u} + 4 |\phi_3|^2 =  h^2.
 \end{equation}
 Therefore completeness of the metric $ds^2$ implies  
 completeness of  the metric $ds_e^2$. Moreover, the Gauss map of 
 the maximal surface always takes values eithre in $\mathbb H^2_+$ or $\mathbb H^2_-$.
 This implies that  the corresponding normal Gauss map (through the stereographic projection to the unit disk $\D$
  or the outside of the unit disk $\C \cup \{\infty\} \setminus \bar \D$) of the corresponding CMC surface takes values in 
 the upper open hemisphere or the lower open hemisphere. But according to \cite[Theorem 1]{hos1982}, there is no  complete non-minimal CMC surface that has Gauss image contained in a hemisphere.
 \end{proof}
 From Theorem \ref{thm:Sym}, it is natural to consider spacelike maximal surfaces in $\Nili$ with 
 singularities. Since the normal Gauss map $g = 
  \psi_1/\overline{\psi_2}$ is a  harmonic map into  $\C \cup 
  \{\infty\}$ minus $\mathbb S^1$, 
 it satisfies the elliptic PDE:
 \begin{equation}\label{eq:harmonicity}
   g_{z \bar z} - \frac{2\bar g}{1+|g|^2} g_{\bar z} g_z=0.
 \end{equation}
 By using the nonlinear Dirac equation 
 in \eqref{Dirac1}, we compute 
 $  \bar g_{z} = - \frac12 (\overline \psi_2)^2 (|g|^2+1)^2$,
 which, by setting 
 $\omega =  i \, \overline{\psi_2}^2$,  is equivalent to
 \begin{equation}\label{eq:gbarz}
 \omega =  -\frac{2i\bar g_{z}}{(|g|^2+1)^2}.
  \end{equation}
 Moreover, by the equation in \eqref{eq:Weingarten} we have
 \begin{equation}\label{eq:Bandg}
    \frac{i }{2}g_z \omega  =  \frac{g_{z} \bar g_{z}}{(|g|^2+1)^2}= B.
 \end{equation}
  The induced metric $ds^2 =4 (|\psi_1|^2- |\psi_2|^2)^2 |dz|^2$ of a maximal surface in $\Nili$ 
  is written in terms of $g$ and $\omega$ as
 \begin{equation}\label{eq:inducedmetric}
  ds^2 = (1-|g|^2)^2 |\omega|^2 |dz|^2 
  =  4|\bar g_{z}|^2\frac{(1-|g|^2)^2}{(1+|g|^2)^4}  |dz|^2.    
 \end{equation}
  Therefore, there are two possibilities to 
  generalize maximal surfaces in $\Nili$ as
 \[
  |g|^2 =1 \quad \mbox{or} \quad g_{\bar z} =0
  \]
  at $p \in M$. 
  From the harmonic map point of view, it is natural to consider a 
  harmonic map into $\C \cup \{\infty\}$. The latter case,  
  the $g_{\bar z}(p)=0$, is equivalent to that the harmonic map $g$ is holomorphic at point $p$. 
  It turns out that in the proof of Theorem \ref{thm:criteria} a holomorphic point is 
  always a degenerate singular point. Therefore we will, generally speaking, exclude such points.
  
  On the other hand, from a nowhere holomorphic 
   harmonic map from a Riemann surface 
  into $\C \cup \{\infty\}$, we can define the generating spinors $\psi_1, \psi_2$ by \eqref{eq:gbarz} and 
  $g = \psi_1/\overline{\psi_2}$. Then one can introduce 
  the extended frame $F(\lambda)$ as in \eqref{eq:extframin} 
  and the Sym-formula $f = f^{\lambda}|_{\lambda=1}$ defined by in \eqref{eq:symNil}. Then it is clear that the map $f$ defines a maximal surface where it has a regular point, that is $|g|^2 \neq 1$. It is natural to define the following class of surfaces.
 \begin{definition}[Generalized spacelike maximal surfaces]\label{def:GSMS}
  Let $g$ be a nowhere holomorphic harmonic map from a 
  Riemann surface $M$ into $\C \cup \{\infty\}$, and 
  a map $f$ into $\Nili$ as defined above.
  Then $f$ will be called a
  {\rm generalized spacelike maximal surface} in $\Nili$, and
  $g$ will be called the {\rm normalized Gauss map} of $f$.
 \end{definition}
 \begin{remark}
 \mbox{}
 \begin{enumerate}
     \item For each $\lambda \in \mathbb S^1$, $f^{\lambda}|_{\lambda \in \mathbb S^1}$ 
      is also a generalized  spacelike maximal surface.
     \item In general, a rotation of a nowhere holomorphic harmonic map $g$ as 
      a map on the $2$-sphere changes the corresponding generalized 
 maximal surface completely; a rotation around the $e_3$-axis $g$ gives 
  an isometry of the surface in $\Nili$ but other rotations do not give isometries. Thus in Definition \ref{def:GSMS} we choose $g$ with some fixed initial condition 
   and choose the extended frame $F$ which gives $\pi \circ N = g$
   with $N = \Ad_F E_3|_{\lambda=1}$.
 \end{enumerate}
 \end{remark}

 %%%%%%%%%%%%%%
\subsection{Numeric Examples} \label{numerics}
 An extended frame for any harmonic map into $\SSS^2$, and hence any maximal surface in $\Nili$, can be constructed via the method of Dorfmeister/Pedit/Wu (DPW) \cite{DorPW}, from a 
 holomorphic potential of the form:
 \[
\xi = \sum_{n=-1}^\infty A_n(z) \lambda^n \dd z.
\]
 Solving the ODE $\Phi_z = \Phi \xi$, with an initial condition $\Phi(z_0) = I$, gives a so-called \emph{complex} extended frame, and an
 extended frame $F$ for a harmonic map is then obtained from $\Phi$
 by a pointwise Iwasawa decomposition $F = \Phi B_+$, where $B_+$ extends holomorphically, in the parameter $\lambda$, to the unit disc. To produce a nowhere
 holomorphic harmonic map $g$, or equivalently, a regular CMC surface
 $f_{cmc}$ we need the regularity condition:
 \beq \label{regularity}
 (A_{-1})_{1,2} \neq 0,
 \eeq
 which corresponds to the function $e^{w/2}$ in $U^\lambda$ being non-zero.

 Potentials corresponding to various geometric properties for CMC surfaces have been found in numerous subsequent works, and the solutions can be computed numerically. 
 
 A major difference when using this method for maximal surfaces in the Heisenberg group is that the action of the group $\SU$ is not an isometry. Pre-multiplying an extended frame by an element of $\SU$ will result in an extended frame for a geometrically distinct solution. Therefore, for a given potential, many different solutions are obtained, depending on the initial condition. 

\textbf{Interpreting the images:} In Figures \ref{figex0} and \ref{figex1}, the surfaces are oriented with the $e_3$-axis directed upwards, and are colored according to whether the harmonic map takes values in the upper or lower hemisphere.
Since the singular set corresponds to the equator separating these two hemispheres,
the plot of the maximal surface in $\Nili$ changes color 
at the singular set. The plot can, of course, appear to change color
at a place where the surface has a self intersection, for example
the lower middle image in Figure \ref{figex1}.

 \begin{example} \label{example0}
  As is well-known, a CMC surface of revolution in $\real^3$ is obtained 
  from a potential of the form:
  \[
   \xi(z) = \bbar 0 & -a\lambda^{-1}+(a-1)\lambda \\ (1-a)\lambda^{-1} + a\lambda & 0 \ebar \dd z.
  \]
  The case $a=1.2$ is computed and plotted in Figure \ref{figex0}.
  The CMC surface is a nodoid. On the second row we show examples of the diferent maximal surfaces in $\Nili$ obtained by choosing a different axis of symmetry for the nodoid. The harmonic Gauss map of the nodoid is doubly periodic. In the figure, for each case, the singular set for the lower surface coresponds to where the normal to the CMC surface above is perpendicular to the $E_3$-direction. At such points, if the Gauss map is symmetric about the equator, then a fold singularity is created on the maximal surface (first two examples).
  \end{example}

\begin{example} \label{example1}
 CMC surfaces with rotationally symmetric metric can be constructed via the potential:
 \[
  \xi(z) =  \begin{pmatrix}
 0  & 1 \\
 z^k & 0
 \end{pmatrix} \lambda^{-1} \dd z, 
 \]
where $k$ is a positive integer. For the case $k=1$, in Figure \ref{figex1} we display a local solution with two different initial conditions for the $SU(2)$-frame. 

 \begin{figure}[htb]
\centering
$
\begin{array}{ccc}
\includegraphics[height=28mm]{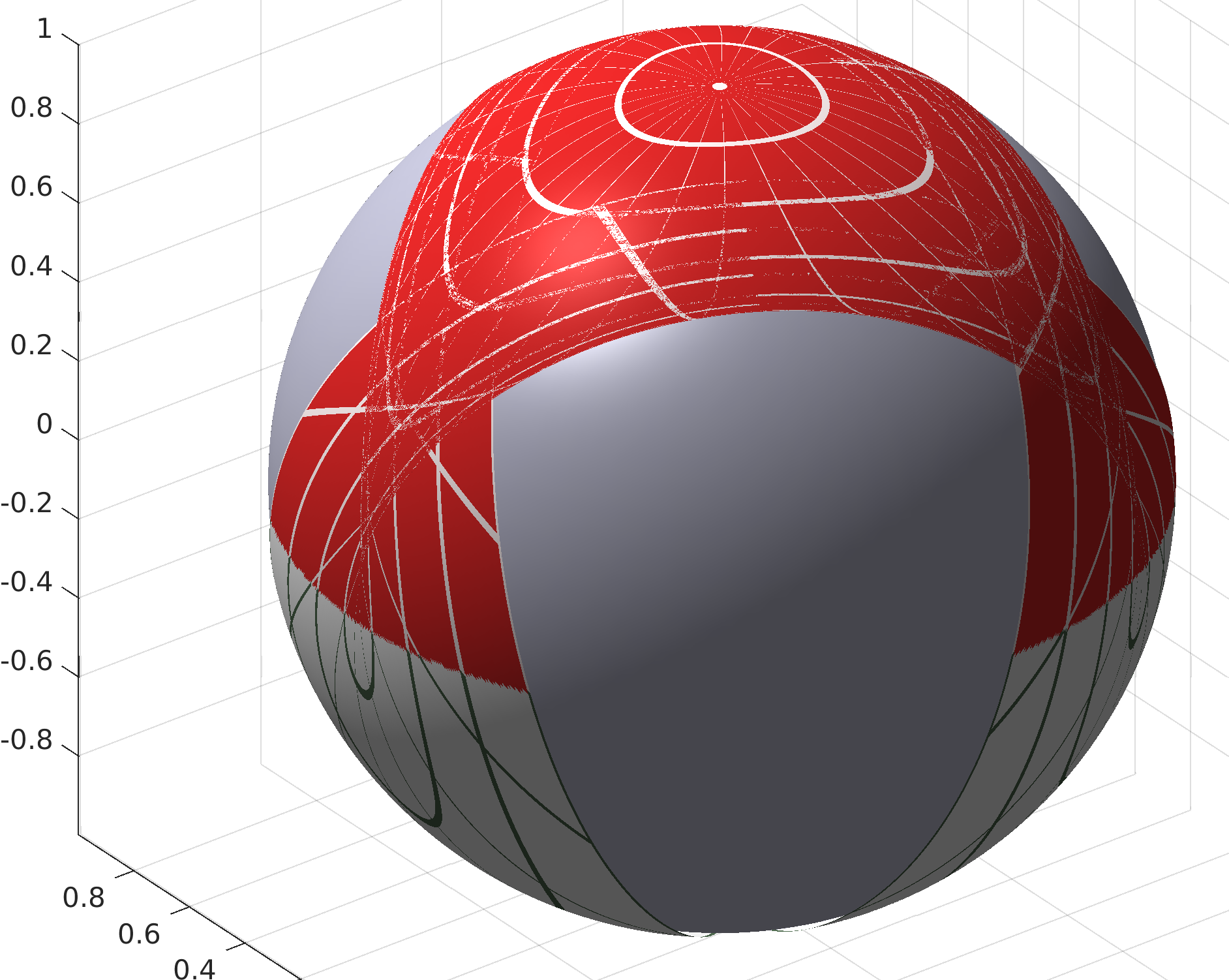}  \,\,
 & \,\,
\includegraphics[height=28mm]{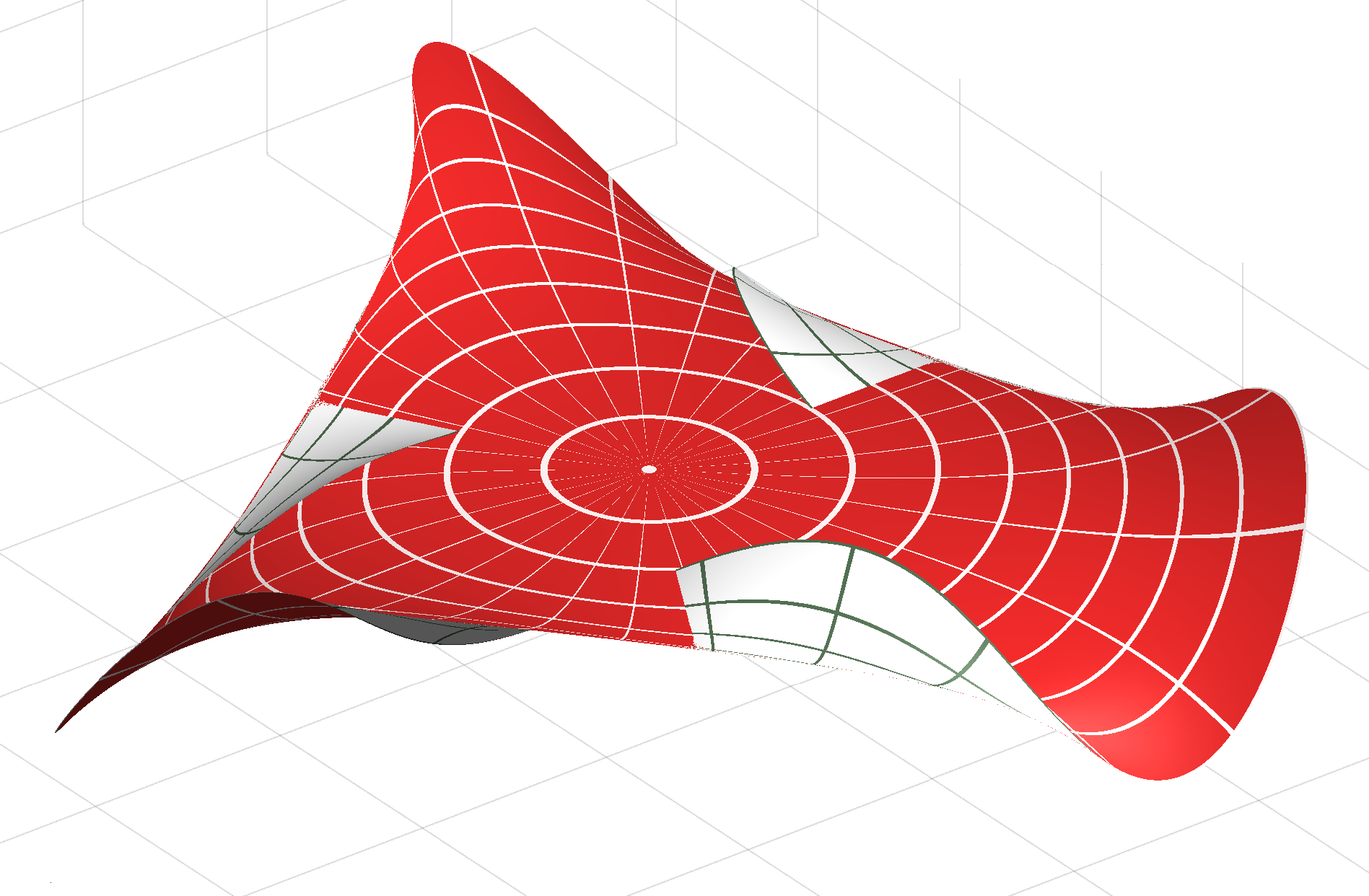} \,\, & \, \,
\includegraphics[height=30mm]{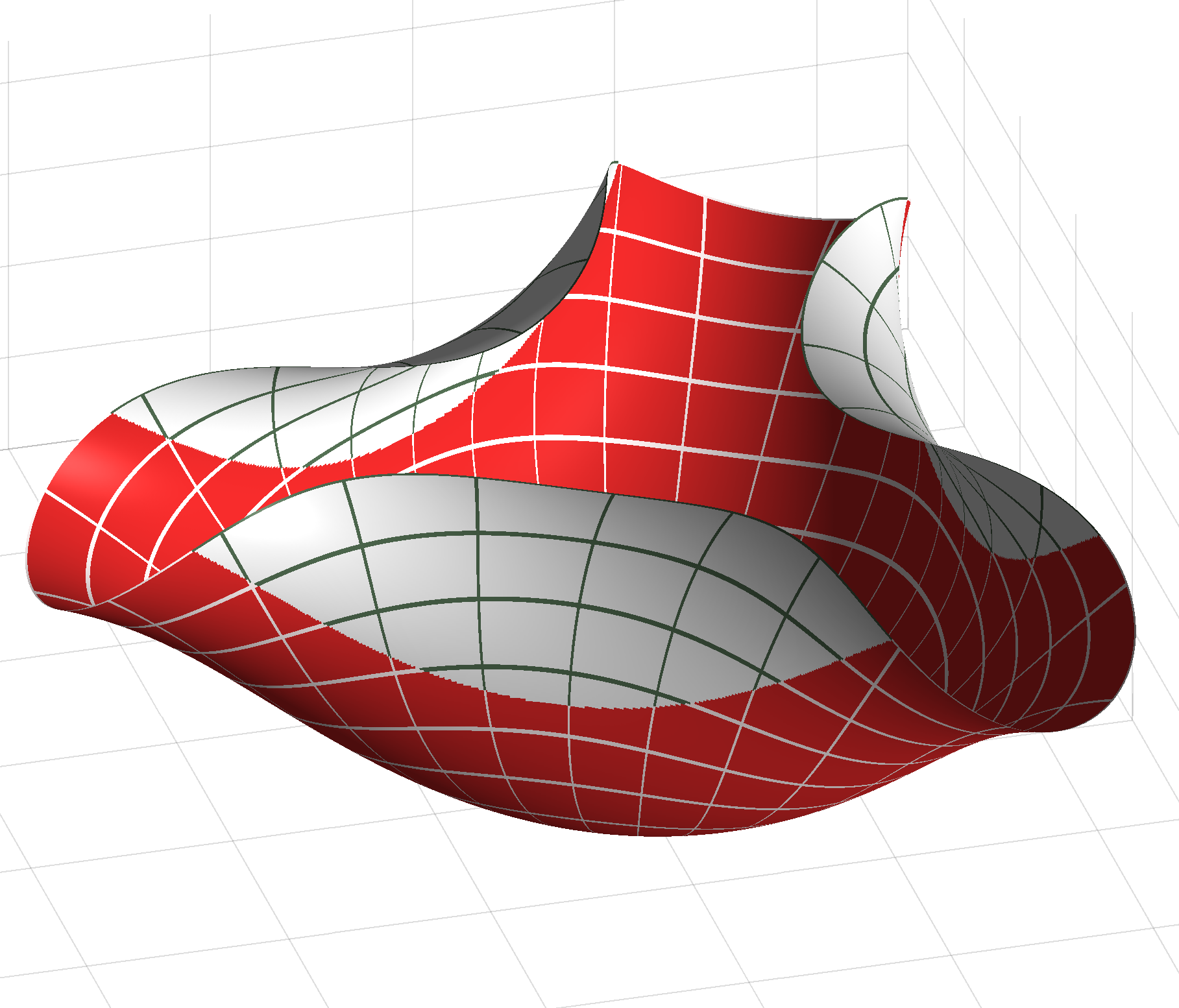}
\vspace{2ex} \\
\includegraphics[height=30mm]{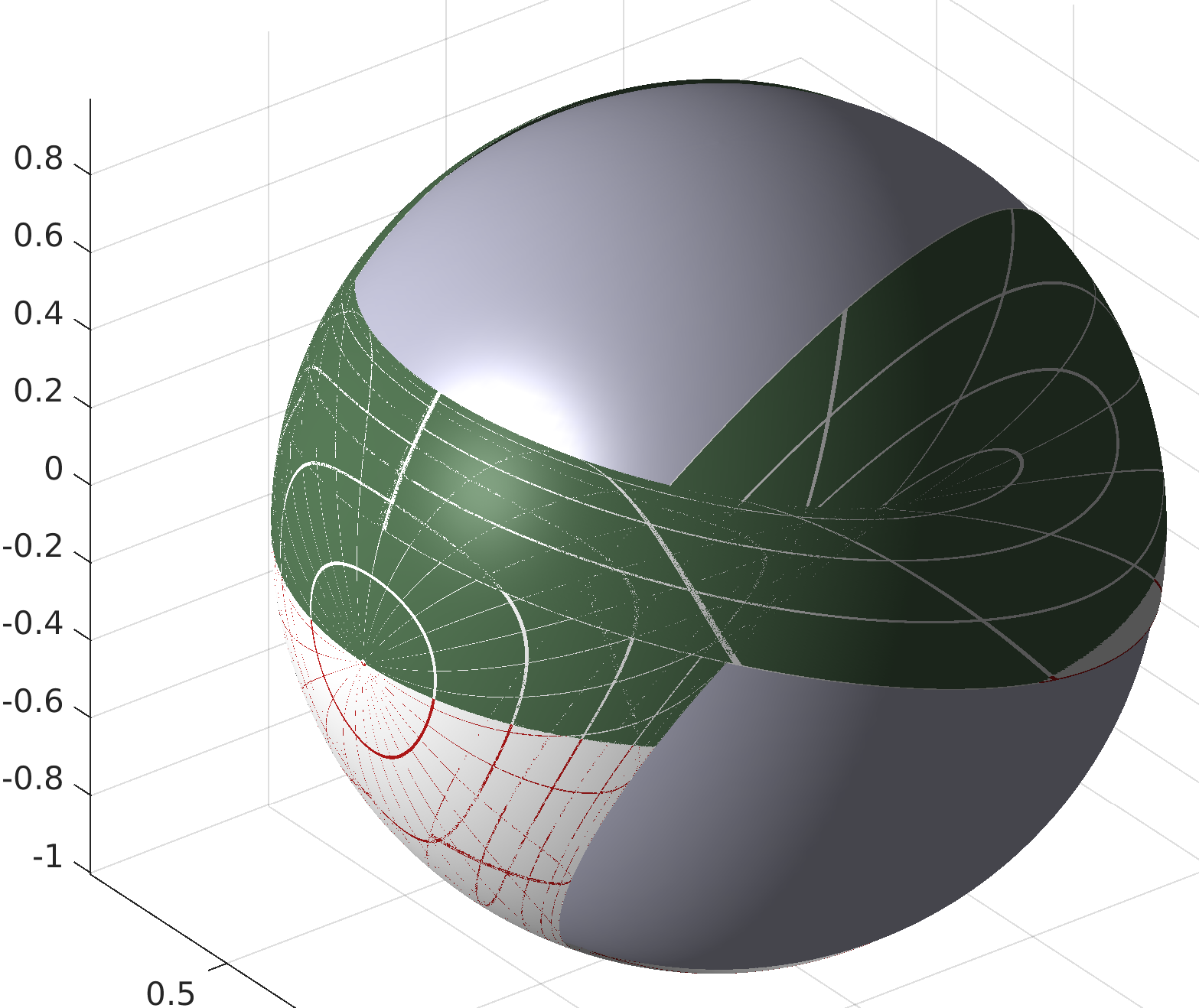} \,\,  & \,\,
\includegraphics[height=35mm]{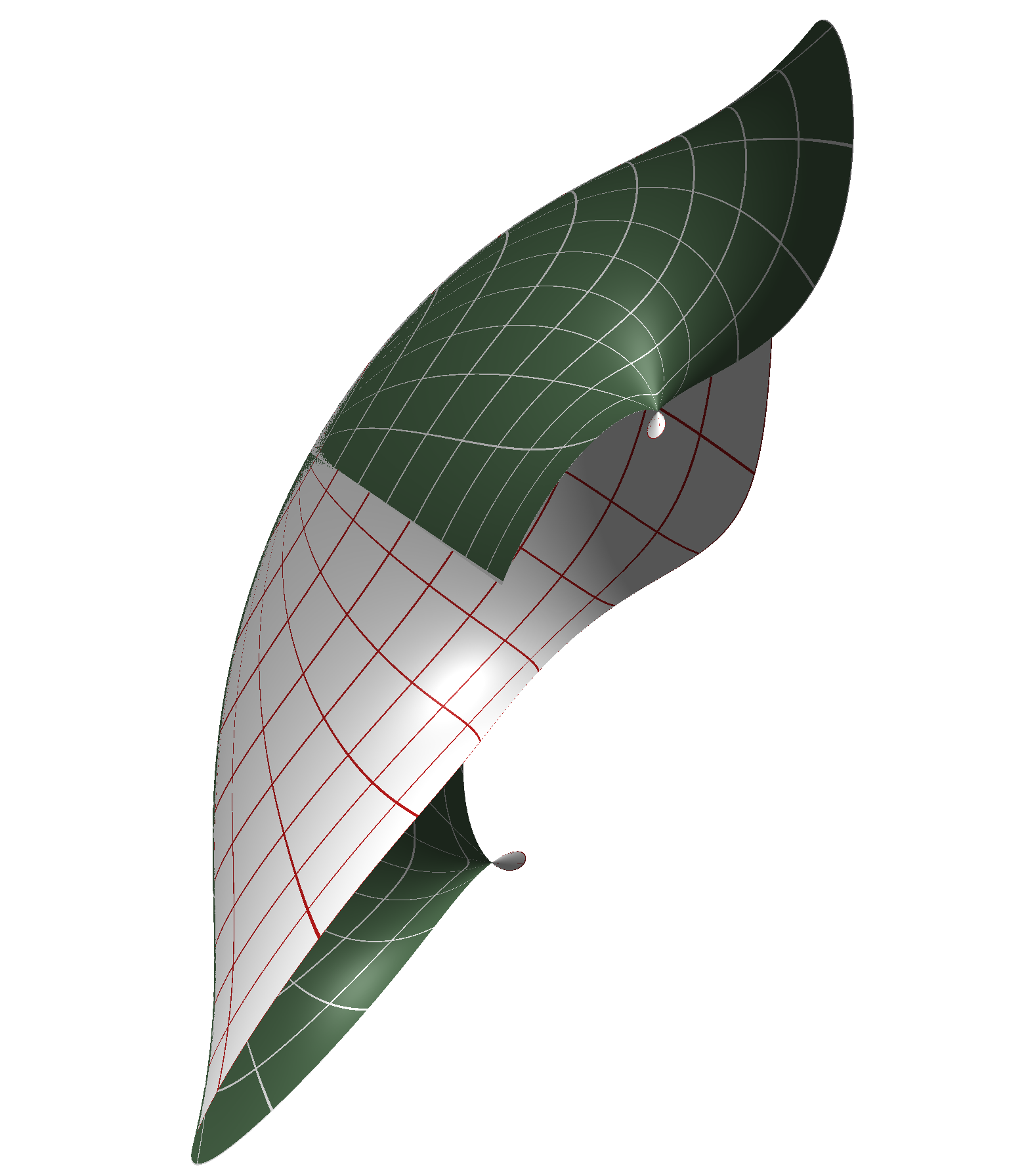}   \,\, & \,\,
\includegraphics[height=30mm]{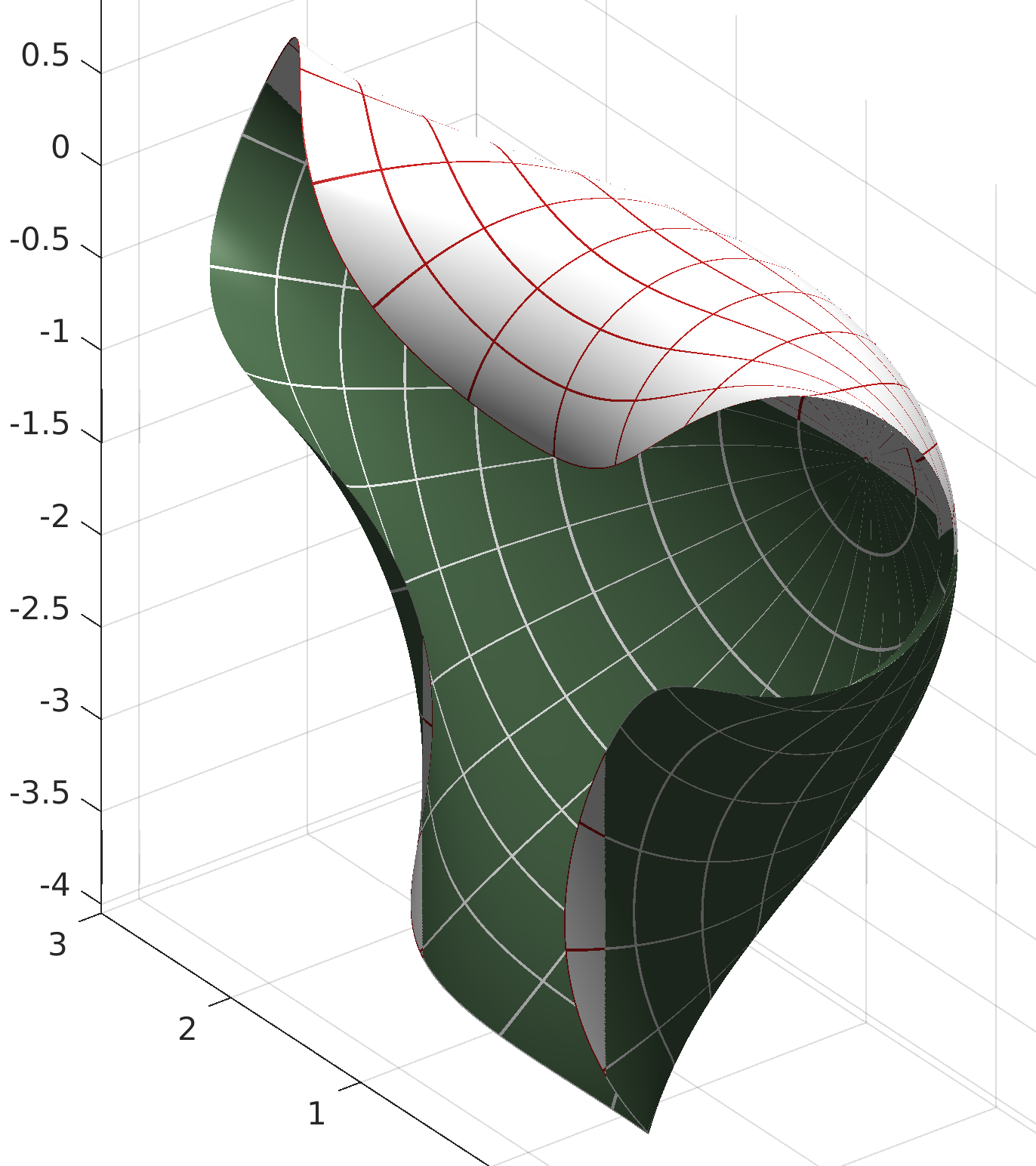} 

\end{array}
$
\caption{The solutions of Example \ref{example1}. Left: harmonic maps; middle: corresponding maximal surface in $\Nili$; right: corresponding CMC surfaces in Euclidean space;   }
\label{figex1}
\end{figure} 

Since $\SU$  acts by isometries on $\SSS^2$ and $\real^3$, the harmonic map and the CMC surface are
unchanged (up to isometry) regardless of the initial condition, but the maximal surface in $\Nili$ is different.  When the $e_3$ axis is chosen as the axis of symmetry for the harmonic map (top row in Figure \ref{figex1}), the maximal surface has an order 3 rotational symmetry about the $e_3$-axis, and has three cuspidal cross-cap singularities within the domain.    When one of the other axes is chosen as the axis of symmetry for $N$ (bottom row), the corresponding maximal surface is no longer symmetric and has only one cuspidal cross-cap singularity in the domain, at the center.
\end{example}

\begin{example} \label{example2}
Potentials of the form
\[
\xi(z) = \begin{pmatrix} 0 & a(z) \\ b(z) & 0 \end{pmatrix} \lambda^{-1} \dd z
\]
where $a(z)$ and $b(z)$ are meromorphic are called \emph{normalized} potentials.  If integrated with initial condition $\Phi(z_0)=I$, then
$z_0$ is called the \emph{normalization point}.
 In \cite{mincmc}, it is
shown (Lemma 7.3) that a CMC  surface has a dihedral rotational symmetry of order $n$ about the normalization point $z=0$
if and only if the meromorphic functions $a$ and $b$ have Laurent expansions of the form:
\[
a(z) = \sum_j a_{nj} z^{nj}, \quad \quad b(z) = \sum_j b_{nj-2} z^{nj-2}.
\]
(Thus the CMC surface in Example \ref{example1} also has a dihedral symmetry of order $k+2$). 
 \begin{figure}[htb]
\centering
$
\begin{array}{cccc}
\includegraphics[height=30mm]{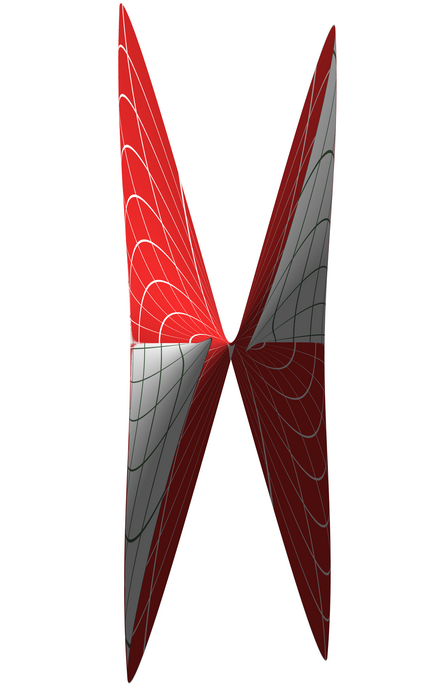}  
 & \,
\includegraphics[height=28mm]{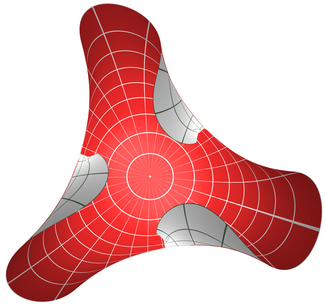} & \, 
\includegraphics[height=28mm]{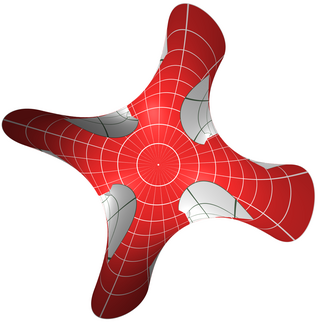} & \, 
\includegraphics[height=28mm]{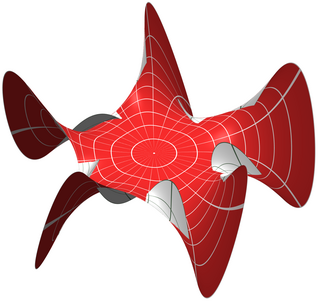} 
\vspace{1ex} \\
(1+z^2,1)  & (1+z^3,z) & (1+z^4,z^2) & (1+z^5,z^3)
\end{array}
$
\caption{
 Maximal surfaces with rotational symmetry 
 (Example \ref{example2}), computed from normalized potentials with the displayed functions $(a(z),b(z))$.}
\label{figex2}
\end{figure} 
Examples are computed and shown in Figure \ref{figex2}, where the initial condition is chosen so that the $e_3$-axis is perpendicular to the tangent plane at the point of symmetry. Thus the maximal surface also has the dihedral symmetry.  For these cases, the order $n$ symmetric examples apparently have $n$ cuspidal cross-cap singularities in the domain computed (a disc around $z=0$).  The corresponding constant positive Gauss curvature surfaces can be found in \cite{spherical}, Figure 4. 
\end{example}

 \section{Generic Singularities for Maximal Surfaces in
 \texorpdfstring{$\Nili$}{Nil3}} \label{singsection}

 The generic singularity criteria for spacelike  maximal faces
 in Minkowski space $\mathbb L^3$
 has been given in \cite[Theorem 3.1]{uy2006}, see also the definition of "maximal faces" there 
 and  in \cite[Theorem 2.4]{fsuy}. Below we will first
 give criteria for three standard singularities for
 maximal surfaces in $\Nili$, and then show that these
 are the generic ones.

 \subsection{Characterization of Standard Singularities}

\begin{theorem}[Singularity criteria]\label{thm:criteria}
 Let $g$ be the normal Gauss map of a generalized spacelike 
 maximal surface $f$ in $\Nili$. Then the singular points of $f$ are given by
$|g|^2 =1$ and $f$ is degenerate at a singular point $p$ if and only if  
\[
\Re \left[\frac{g^{\prime}}{g^2 \omega }\right] = \Im \left[\frac{g^{\prime}}{g^2 \omega}\right] +2 =0,
\]
where $\prime=\partial_z$ and the function $\omega$ is defined by \eqref{eq:gbarz}, i.e., $\omega = -(2i\bar g_{z})/(|g|^2+1)^2$.  Moreover, $f$ is a wave front at a singular point $p$ if and only if
 $\Re (g^{\prime}/(g^2 \omega)) \neq 0$,  and 
  the following criteria hold$:$
  \begin{enumerate}
     \item $f$ is $\mathcal A$-equivalent to a cuspidal edge at $p$  if and only if 
     \[
     \Re \left[ \frac{g^{\prime}}{g^2 \omega}\right] \neq 0
     \quad \mbox{and}\quad \Im \left[ \frac{g^{\prime}}{g^2 \omega}\right] \neq -2.
     \]
      \item $f$ is $\mathcal A$-equivalent to a swallowtail at $p$  if and only
      if     
      \[
      \frac{g^{\prime}}{g^2\omega} + 2 i  \in \mathbb R^{\times}\quad \mbox{and}\quad 
      \Im \left[\frac1{(\log |g|)^{\prime}}\left(\Im \left(\frac{g^{\prime}}{g^2 \omega}\right)\right)^{\prime} \right] \neq 0.
      \]
    \item $f$ is $\mathcal A$-equivalent to a cuspidal cross-cap at $p$  if and only if
      \[
     \frac{g^{\prime}}{g^2\omega} + 2 i  \in i \mathbb R^{\times}
       \quad \mbox{and}\quad 
            \Re \left[\frac1{(\log |g|)^{\prime}}\left(\Im \left(\frac{g^{\prime}}{g^2 \omega}\right)\right)^{\prime} \right]
            \neq 0.  \]
\end{enumerate}
\end{theorem}
\begin{proof}
 Since the induced metric is given by \eqref{eq:inducedmetric}, 
 it is easy to see that $f$ has a singular point at $p$  if $|g|^2=1$.
 Moreover, around a point $p$ where $g(p) = \infty$
 the induced metric can be computed as
 \[
  ds^2 =4|\bar g_{z}|^2\frac{(1-|g|^2)^2}{(1+|g|^2)^4}  |dz|^2
  = 4|\bar {\hat g}_{z}|^2\frac{(1-|\hat g|^2)^2}{(1+|\hat g|^2)^4}  |dz|^2,
 \]
  where $\hat g = 1/g$. The normal Gauss map is nowhere holomorphic and thus the induced metric at the point $g(p) = \infty$ is regular.
  Thus the first claim follows.
  
  The Euclidean cross product $\times_{e}$ of $f^{-1}f_x$ and $f^{-1}f_y$, where $z = x + iy$, is given by 
\begin{align*}
f^{-1}f_x \times_{e} f^{-1}f_y &= -2 i f^{-1}f_{z} \times_{e} f^{-1}f_{\bar z} \\
 &=  |\omega|^2  (|g|^2 - 1)\left( 2 \Im (g )e_1- 2 \Re (g)e_2 + (|g|^2 + 1)e_3\right),    
\end{align*}
 where $\omega = -(2i\bar g_{z})/(|g|^2+1)^2$ as before.
  Note that at a singular point we have,  $\omega (p) =- i \bar g_{z}(p)/2 \neq 0$.
 Then the Euclidean unit normal $f^{-1} N_{e}$ is computed by 
 \[
f^{-1} N_{e} = \frac1{\sqrt{(1+|g|^2)^2 + 4 |g|^2}}\, \left(2 \Im(g) e_1 -2 \Re(g)e_2 + (1+|g|^2)e_3
\right).    
\]
Thus the function $\rho$ defined by $f^{-1}f_x \times_e f^{-1}f_y = \rho f^{-1} N_e$ is 
given by 
\[
 \rho = (|g|^2-1) |\omega|^2 \sqrt{(1+|g|^2)^2 + 4 |g|^2}.
\]
A singular point $p \in M$ is called \emph{non-degenerate} if $d \rho$ does not
vanish at $p$ (see \cite{uy2006}) and otherwise called \emph{degenerate}.  It is easy to see by 
 using $|g|^2 =1$ that
\[
 \dd \rho (p)  = 2 \sqrt{2} |\omega|^2 \left(\frac{\dd g}{g} + \frac{\dd \bar g}{\bar g}\right).
\]
Since $\omega= - i \bar g_{z}(p)/2$ is assumed non-vanishing, 
$\dd \rho(p)$ vanishes if and only if $(\dd g/g + \dd \bar g/\bar g$ vanishes,
which is equivalent to
$g'/(g^2 \bar g') + 1 =0$.
 Therefore $f$ is degenerate at a singular point $p$ if and only if 
 \[
 \Im (g^{\prime}/(g^2 \bar g^{\prime} )) = \Re (g^{\prime}/(g^2 \bar g^{\prime})) + 1 =0,
 \]
which is equivalent to the first conditions.
 
 Next using $g$ and $\omega$ at a singular point $p$, that is $|g(p)|^2=1$ or $\bar g(p) = 1/g(p)$,
we have
\[
f^{-1} d f =(g \omega \dd z + \bar g \bar \omega \dd \bar z) (\Im (g) e_1  - \Re (g)e_2 - e_3).
\]
Thus $\eta = i/(g \omega)$ is the null-direction of $df$ at $p$.
A straightforward computation shows that 
\[
d N_e (p) = -\frac{i}{2\sqrt{2}} \left(\frac{d g}{g}- \frac{d \bar g}{\bar g}\right)
(\Re(g), \Im(g), 0).
\]
Then the null-direction of $d N_e$ at $p$ is proportional to 
\[
 \mu = \overline{\left( \frac{g_z}{g}\right)}-\frac{g_{\bar z}}{g}.
\]
It is known that $f$ is a wave front \cite{krsuy} if and only if
$\det (\mu, \eta) \neq 0$, that is, 
\[
\det (\mu, \eta)= \Im (\bar \mu \eta) = \Re \left\{\left( \frac{g_z}{g}- 
\frac{\bar g_{z}}{\bar g}\right) \frac{1}{g \omega}\right\}\neq0.
\]
Plugging the expression in \eqref{eq:gbarz} into this, we have
the criterion for a wave front. We now assume that
$f$ is a wave front at $p$, that is, $\Re ( g_z/(g^2 \omega) )\neq 0$ at a singular point $p$.
Then the singular curve $\gamma(t)$ with $\gamma(0)=p$ satisfies
$g(\gamma(t)) \overline{g(\gamma(t))}=1$, and thus 
\[
\Re \left( \frac{g_{z}}{g} \dot \gamma +  \frac{g_{\bar z}}{g} \dot {\bar \gamma} \right) = 0
\]
on the singular curve $\gamma (t)$. Then without loss of generality, $\dot \gamma$ is given by 
\begin{equation}\label{eq:singder}
\dot \gamma(t)= i \left(\overline{\left(\frac{g_{z}}{g}\right)}  +  \frac{g_{\bar z}}{g}\right)(\gamma(t)).
\end{equation}
Then using the criterion in \cite[Proposition 1.3 (1)]{krsuy}, $f$ has a cuspidal edge 
 at $p$ if and only if 
\[
\det (\dot \gamma, \eta)|_{t=0}= \Im (\dot{\bar \gamma} \eta )
=  \Im \left(\frac{g_z}{g^2 \omega} + i \frac{(|g|^2+1)^2}{2 |g|^2}\right)
= \Im \left(\frac{g_z}{g^2 \omega}\right)+2 \neq 0,
\]
 where we use $|g|^2=1$. Now the relation $\omega(p) = -i \bar g_{z} (p)/2$ implies the conclusion.
 
 We now assume $\Re ( g_z/(g^2 \omega) )\neq 0$ and 
 $\Im \left(g_z/(g^2 \omega)\right)= -2$. Then using the criterion in \cite[Proposition 1.3 (2) ]{krsuy}, $f$ has a 
 swallowtail at $p$ if and only if 
 \[
 \frac{d \det (\dot \gamma, \eta)}{dt}\Big|_{t=0} \neq 0.
 \]
 Then a straightforward computation by using a relation
 $\frac{d}{dt}|_{t=0} (|g|^2+1)^2|g|^{-2}=0$ on $|g|=1$ shows that
 \[
 \frac{d \det (\dot \gamma, \eta)}{dt}\Big|_{t=0}
 = \Im \left\{\left(\frac{g_z}{g^2 \omega}\right)_z \dot \gamma
 + \left(\frac{g_z}{g^2 \omega}\right)_{\bar z} \dot {\bar\gamma}
 \right\}.
 \]
 Inserting the expressions in \eqref{eq:singder} into the above equation, we have the second 
 criterion. 

Finally we consider a criterion for cuspidal cross-caps. It is known that
\cite[Corollary 1.5]{fsuy}, $f$ at $p = \gamma (0)$ is $\mathcal A$-equivalent to 
a cuspidal cross-cap  if and only if
\begin{enumerate}
    \item $\eta (0)$ is transversal to $\dot \gamma(0)$.
    \item $\chi (0) =0$ and $\dot \chi (0)\neq 0$.
\end{enumerate}
 Here $\chi$ is defined by (see the proof of \cite[Theorem 2.4]{fsuy})
 \[
  \chi =  \det (f_* \dot \gamma, d N_e (\eta), N_e).
  \]
It is easy to verify that the condition (1) can be computed by 
\[
 \det (\dot \gamma, \eta) = 
  \Im (\overline{\dot \gamma}  \eta)  = \Im \left(\frac{g_z}{g^2 \omega}\right)+2 \neq 0
\]
 at $p$. Then the function $\chi$ can be computed as 
 \[
  \chi = \Re \left( \frac{g^{\prime}}{g^2 \omega}\right)  \chi_0,
 \]
  where $\chi_0$ is a smooth function and it is non-vanishing at $p$.
  Thus the condition (2) can be computed by 
\[
\Re \left( \frac{g^{\prime}}{g^2 \omega}\right) = 0, \quad 
\Re \left\{\left(\frac{g_z}{g^2 \omega}\right)_z \dot \gamma
 + \left(\frac{g_z}{g^2 \omega}\right)_{\bar z} \dot {\bar\gamma}
 \right\} \neq 0
 .
\]
  Inserting the expressions in \eqref{eq:singder} into the above equation, we have the third
 criterion. 
\end{proof}
\begin{remark}
 The normal Gauss map $g$ for a generalized spacelike maximal surface in $\Nili$
 is nowhere holomorphic harmonic, that is $g_{\bar z}$ is never zero.
 Therefore, the criteria in Theorem 
 \ref{thm:criteria} is different from that in 
 \cite[Theorem 3.1]{uy2006} and  in \cite[Theorem 2.4]{fsuy}, where 
 the normal Gauss map $g$ was holomorphic.
\end{remark}

\subsection{A simplified characterization}

 We now want to simplify the criteria in Theorem \ref{thm:criteria} 
 using a normalization of the normal Gauss map $g$ at a singular point.
 First we observe:
  \begin{lemma}\label{lem:phi3der}
 Let $g$ be the normal Gauss map of a generalized spacelike maximal surface, and $\omega = -2i\bar g_z/(|g|^2+1)^2$.
  At a non-holomorphic singular point we can choose a coordinate $z$ such that:
  \beq  \label{eq:normalization}
  \dd (g \omega)(p)  =0,  \quad \quad \quad  (g\omega)(p) = 1.
  \eeq
\end{lemma}
\begin{proof}
Using the harmonic map equation \eqref{eq:harmonicity} we obtain:
\[
  \dd (g \omega)(p)  =\frac{i}{2}g\left(\bar g_{z \bar z} - g \bar g^2_z\right) \dd z.
  \]
Since $\omega$ depends on the choice of coordinates, we can in fact
choose coordinates to make the above expression zero. The new coordinate
$\zeta$ is found by solving the equation
$\dd^2 z/\dd \zeta^2 + C (\dd z/\dd \zeta)^2 = 0$, where $C$ is the value at $p$ of $(\bar g_{zz}+g(\bar g_z)^2)/\bar g_z$.
Since $(g \omega)(p) \ne 0$, a constant scaling of the coordinate finishes the proof.
\end{proof}

  Recall that by \eqref{eq:Bandg}, the Abresch-Rosenberg differential $B \, \dd z^2$, the normal Gauss map $g$,
  and the function $\omega$ have the following relation
 \begin{equation*}
  \frac{i}{2}g_z \omega = B.
 \end{equation*}
 Putting these expressions into Theorem \ref{thm:criteria} results
 in the following simplified statement:
 %%%%%%%%%%%%%%%%%%%%%%%%%%
 \begin{theorem} \label{thm:criteria2}
 Let $g$ be a normal Gauss map of a generalized maximal surface $f$
 in  $\Nili$, with coordinates chosen as in
 Lemma \ref{lem:phi3der} at a singular point $p$.
 Then the singular point is
 degenerate if and only if 
 \[
 \Im \left( B\right) =  0 \quad \hbox{and} \quad \Re \left(B\right) =1 \]
 at $p$, and
 $f$ is a wave front at  $p$ if and only if
 $\Im (B) \neq 0$.
 Moreover, writing  $\prime = \partial_z$,
 $f$ is $\mathcal A$-equivalent at $p$ to a:
  \begin{enumerate}
     \item  cuspidal edge   if and only if $\Re (B) \neq  1$ and $\Im (B) \neq 0$;
      \item swallowtail
       if and only if $\Re (B) =  1$ and $\Im (B) \neq 0$
      and $\Im \left(B^{\prime}\right) \neq 0$;
    \item cuspidal cross-cap  if and only if
    $\Re (B) \neq  1$,  $\Im (B) = 0$  and
     $\Im \left(B^{\prime}\right) \neq 0$.
\end{enumerate}
\end{theorem}

\begin{remark}\label{rm:hgomega}
 Vanishing of the derivative $h_z$ of the support function
 at a singular point 
 is related to the vanishing of $(g\omega)_z$. In fact the support function  $h$
 can be rephrased as 
 \[
 h^2 = 4 |\omega|^2(1+|g|^2)^2,
 \]
  and a straightforward computation, using the
  relations $|g(p)|^2=1$ and  $(g\omega)_{\bar z}(p)=0$ 
   at a singular point $p$, implies that
   $h_z(p)=0$ is equivalent to $(g\omega)_z(p)=0$.
\end{remark}

%%%%%%%%%%%%%%%%%%%%%%%%%%%%%%%%%%%%%%%%%%%%%%
\subsection{A construction of singularities through the DPW representation} \label{dpwsect}

We now show how to obtain these singularities via the DPW method described in
Section \ref{numerics}.
A general holomorphic potential for the DPW method is of the form:
\begin{equation}\label{eq:xiz}
 \xi(z) = \sum_{n=-1}^\infty \xi_i(z) \lambda^i \dd z,   
\end{equation}
where $\xi_i$ is holomorphic and takes values in $\mathfrak{sl}(2,\C)$, and is diagonal for even values of $i$ and off-diagonal for odd.
 If $\xi$ satisfies the regularity condition \eqref{regularity}, we can write:
 \begin{equation}\label{eq:generalpot}
\xi_{-1} := \bbar 0 & a(z) \\ b(z) & 0 \ebar, \quad a(0) \neq 0,
\quad
B(z) := -b(z)a(z).
\end{equation}
We now want to construct a generalized maximal spacelike surface in $\Nili$ which 
 has a singularity  at $p$.
 Let $C$ be a solution of $d C = C \tilde  \xi$ taking values in $\LSL$ and 
 choose an initial condition 
\begin{equation}\label{eq:initial}
 C(p, \lambda) =    
 C_0:= \frac1{\sqrt{2}} 
     \begin{pmatrix}
         e^{i c} & i e^{i c} \lambda \\ i e^{-i c}\lambda^{-1}& e^{-i c}
     \end{pmatrix}\in \LSU, \quad c \in \mathbb R,
\end{equation}
 that is, at $p$ the Iwasawa decomposition is trivial $C(p, \lambda) = F(p, \lambda)$ and $V_+(p, \lambda) = \mathrm{id}$.  
 Since $i g \omega$ is given by the product of 
 $(1,1)$- and $(2,1)$- of $F$,  the 
 map $g \omega$ takes value $1$ at $p$ and  we have:
 %%%%%%%%%%%%
  \begin{theorem}\label{thm:singDPW}
 Let $\xi$ be the potential given by \eqref{eq:generalpot} and $C$ be the solution of $d C = C \xi$ with initial condition 
  \eqref{eq:initial}. Then the resulting generalized spacelike maximal 
 surface $f$ in $\Nili$ through the DPW method has a singular point at 
 $p(=0) \in \D$ and the quadratic differential
  $B \, dz^2$ in \eqref{eq:generalpot} becomes the Abresch-Rosenberg differential for $f$.
\end{theorem}
\begin{proof}
 From the standard procedure of the 
 DPW method, an extended frame $F$
 which takes values in $\LSU$ can be constructed by 
 the Iwasawa decomposition of 
 $C$, that is 
  $C = F B_+$, where $C$ is the solution of 
 $\dd C = C \xi$ with the initial 
  condition $C(0)=C_0$. Note that $B_+$ takes values in $\LSLP$.  Since 
  the initial condition takes values in $\LSU$, thus the $B_+$ is 
  identity at $p=0$ and thus $C_0= F$
   at $0$. Thus the Gauss map $g$
  is $i$ at $0$ and therefore 
  the resulting surface $f$ has
  a singular point at $0$.
  Then a direct computation shows that 
  \[
  F^{-1} F_{z} = B_+ (\xi/\dd z) B_+^{-1}
  - B_{+ z} B_+^{-1}
  \]
  holds. Comparing it to $U^{\lambda}$ in \eqref{eq:alpha}, $B(z)$
   in the potential in 
   \eqref{eq:generalpot} is the coefficient function of the Abresch-Rosenberg differential.
   \end{proof}
%%%%%%%%%%%%%%%%%%%%%%%
 We cannot immediately apply the criteria of Theorem \ref{thm:criteria2} to $B(z)$ for an arbitrary potential, because we also need the normalization conditions
 \eqref{eq:normalization} to hold. In fact any harmonic map can be represented by a \emph{normalized} potential $\xi = {\tiny{\bbar 0 & a(z) \\ b(z) & 0 \ebar}} \lambda^{-1} \dd z$, as in Example \ref{example2}.
 After a change of coordinate
  $\int_0^z a(w) \dd w$, these are included in:
  %%%%%%%%%%%%%%%%%%%%%%%%%%%
   \begin{corollary}\label{cor:singnormlDPW}
   Any holomorphic potential of the form:
    \begin{equation}\label{eq:normalizedinw}
   \xi =
\begin{pmatrix}
 0  & 1 \\
- B(z) & 0
\end{pmatrix} \lambda^{-1} \dd z
 + \sum_{n=1}^\infty \xi_i(z) \lambda^i \dd z
\end{equation}
satisfies the
  normalization given in \eqref{eq:normalization}. 
 The maximal surface $f$ obtained by integrating $\xi$ with the initial condition \eqref{eq:initial} has a
 singularity at the point $z=0$, and
 the singularity criteria in Theorem \ref{thm:criteria2}
 can be applied to the coefficient function $B(z)$.
   \end{corollary}
%%%%%
   \begin{proof}
   From Theorem \ref{thm:singDPW}, $f$ has a singular point at $0$. 
   Let $C$ be the solution of $\dd C = C \xi$ with
   the initial condition \eqref{eq:initial} and consider 
   the Birkhoff decomposition of $C$ as
\[
C = C_- C_+,
\]
where $C_- \in \LSLN, C_+ \in \LSLP $. Then $C_-$ gives 
\begin{equation}\label{eq:xihat}
\hat \xi = C_-^{-1} d C_- = \begin{pmatrix}
 0  & 1 \\
- B(z) & 0
\end{pmatrix} \lambda^{-1} \dd z,
\end{equation}
i.e. $\hat \xi$ is a normalized potential. Now consider the
Iwasawa decomposition of $C_-$ as 
\[
C_- = F V_+,
\]
where $F \in \LSU, V_+ \in \LSLP$. A straightforward computation 
shows that 
\[
F^{-1} F_{z} = V_+ \hat \xi V_+^{-1} - \dd V_+ V_+^{-1}
\]
holds. Note that $F^{-1} F_{z} = U^{\lambda}$ 
in \eqref{eq:alpha} and it is given by 
\[
U^{\lambda} = 
\begin{pmatrix}
 \frac12 (\log h)_z & \frac14 h \lambda^{-1} \\
 -  4 B h^{-1}\lambda^{-1} & -\frac12 (\log h)_z 
\end{pmatrix},
\]
 where $h$ is the support function and $B$ is the 
 Abresch-Rosenberg differential.
Now taking the holomorphic part of both sides, see 
for an example \cite{Wu1999},  gives
\[
\frac12 (\log\hat h)_z = - (\log \hat v)_z, \quad 
\frac14 \hat h = \hat v^2,
\]
 where we denote the holomorphic part of 
 $h$ by $\hat h$ and the holomorphic part of $V_+$ by  
 $\hat V_+ = \di (\hat v, \hat v^{-1})$. From these equations
 $\hat h$ is constant and in particular $\hat h_z =0$.
  Moreover, at a singular point $z = 0$, $\hat h_z (z=0) 
   = h_z(z=0, \bar z=0)=0$ holds. Finally by the Remark \ref{rm:hgomega}
 it is equivalent to $\partial_z (g\omega)(0)=0$.
 Since $F$ at $z=0$ is given by \eqref{eq:initial},
 $(g\omega)(0)=1$ clearly holds,
 and thus the resulting 
 surface $f$ satisfies the normalization \eqref{eq:normalization}
 at $0$. 
\end{proof}
%%%%%%%%%%%%%%
 
    \begin{figure}[htb]
\centering
$
\begin{array}{ccc}
\includegraphics[height=30mm]{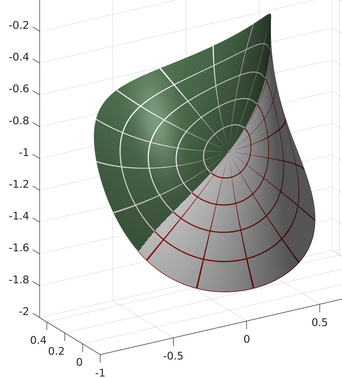} \,\, 
 & \quad
\includegraphics[height=30mm]{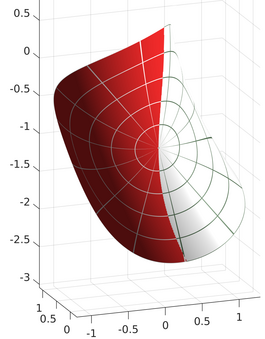} \quad & \, \,
\includegraphics[height=30mm]{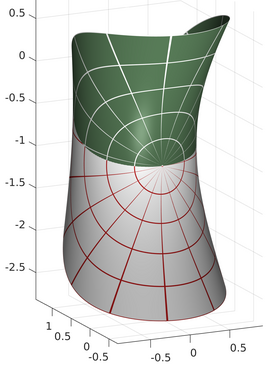} 
\vspace{1ex} \\
\includegraphics[height=24mm]{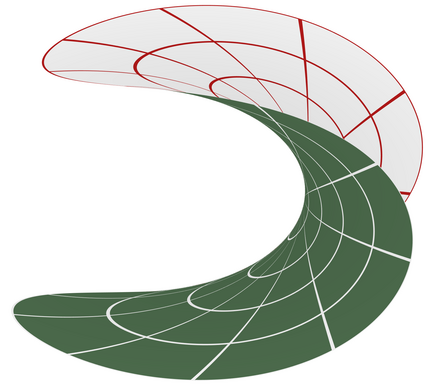} \,\,  & \quad
\includegraphics[height=24mm]{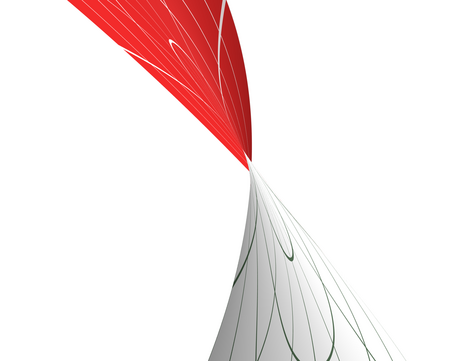}   \quad & \,\,
\includegraphics[height=24mm]{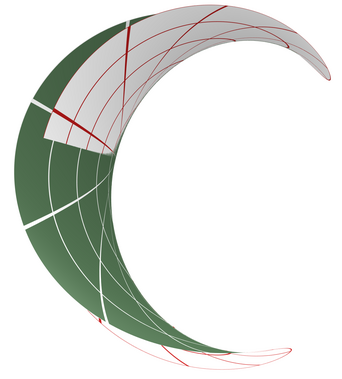} 
\end{array}
$
\caption{Top: CMC $1/2$ surfaces in $\E^3$. Bottom: the corresponding maximal surfaces in $\Nili$, with, in order, cuspidal edge, swallowtail, and cuspidal cross-cap singularities. (Example \ref{singexample}). }
\label{figsing}
\end{figure} 

\begin{example} \label{singexample}
 Local solutions for the potential and initial conditions given above at \eqref{eq:normalizedinw} and \eqref{eq:initial} (with $\xi_i = 0$ for $i>-1$) are computed for the cases
 \begin{enumerate} \item $B= 2-i$: cuspidal edge at $z=0$;
 \item $B=1-i-i z$: swallowtail at $z=0$;
 \item $B=2- i z$: cuspidal cross-cap at $z=0$;
  \end{enumerate}
 and displayed in Figure \ref{figsing}. Note: let $Eq$ denote the equator in $\SSS^2$ perpendicular to the $e_3$-axis, and set $C:=N^{-1}(Eq)$, and $C_{cmc} := f_{cmc}(C)$, i.e., $C_{cmc}$ is the set on $f_{cmc}$ corresponding to the singular set on then
 maximal surface.
 Then, in the images, $C_{cmc}$ is parallel exactly to the $e_3$-axis at the swallowtail singularity, and  perpendicular to the $e_3$-axis at the cuspidal cross-cap (see Theorem \ref{thm:equatorial} below).
\end{example}

\subsection{The generic singularities for maximal surfaces}
To understand the generic singularities for maximal surfaces, first observe that a non-holomorphic harmonic map $N$ through $p \in \SSS^2$
is precisely given by an initial condition $F_0 \in SU(2)$, plus a normalized potential form: $\xi = \begin{pmatrix}
 0  & 1 \\
- B(z) & 0
\end{pmatrix} \lambda^{-1} \dd z$.  The map $N$ is obtained as
$N= \Ad _F E_3$, whre $E_3 = \di(i -i)$ and $F$ is the $SU(2)$ frame obtained from the potential $\xi$ via the DPW method, integrating with the initial condition $\Phi(z_0)=F_0$. The frame $F$ for $N$ has the freedom of right multiplication by any map in into the diagonal subgroup $K$, thus we can assume that the initial condition $F_0$ is of the form:
\[
 F_0 = \begin{pmatrix} r & \sqrt{1-r^2}e^{i c} \\
         -\sqrt{1-r^2}e^{-ic}  & r
       \end{pmatrix}, \quad \quad r \in [0,1], \quad c \in \real /2\pi \Z.
\]
Conversely, given a local harmonic map $N$, an $SU(2)$ frame can be chosen with an initial condition of the above form, and the normalized potential
obtained via the Birkhoff decomposition is unique and, after a change of coordinates, has the form of $\xi$ above.

Thus the data for a local solution for a non-holomorphic harmonic map, and hence for a local generalized maximal surface $f$ in $\Nili$ is the triple $(r,c, B(z)) \in [0,1] \times \real /2\pi \Z \times \mathcal{O}(0)$, where $\mathcal{O}(0)$ denotes the
vector space of germs of holomorphic functions at $0$. The value of $c$ has no geometric significance, since changing $c$ amounts to a rotation about the $e_3$ axis, an isometry of both $\SSS^2$ and of $\Nili$. We therefore, without loss of generality, take $c=0$ and regard the space of local solutions as being
$[0,1] \times \mathcal{O}(0)$.  Additionally, $N$ takes values in the unit circle in the $e_1e_2$ plane (and hence $f$ is singular) at $z=z_0$ if and only if $r = 1/\sqrt{2}$.
Setting $\delta(z) = B(z)-1$ in Theorem \ref{thm:criteria2} we obtain:
\begin{theorem} \label{thm:generic}
 The set of local generalized maximal surfaces $f$ with $f(0)=p$ is in one to one correspondence with the set
 $[0,1] \times \mathcal{O}(z_0)$.  If $(r,\delta)$ is the data for a local solution $f$, then $f$ is singular at $z=z_0$ if and only if $r=1/\sqrt{2}$.
A singularity at $z_0$ is non-degenerate if and only if $\delta(z_0) \neq 0$.
A non-degenerate singularity at $z_0$  is $\mathcal{A}$-equivalent to a:
 \begin{enumerate}
  \item  cuspidal edge
  if and only if $\Re \delta(z_0) \neq 0$ and $\Im \delta(z_0) \neq 0$.
  \item  swallowtail  if and only if $\Re(\delta(z_0)) = 0$, $\Im \delta(z_0)  \neq 0$, $\Im \delta'(z_0) \neq 0$,
  \item  cuspidal cross-cap  if and only if  $\Im \delta (z_0)  = 0$,
  $\Re(\delta(z_0)) \neq 0$ and $\Im(\delta '(z_0)) \neq 0$.
 \end{enumerate}
 \end{theorem}
 Since the local solution $f$ through $p$ obtained from a pair $(r,\delta) \in [0,1] \times \mathcal{O}(z_0)$ depends analytically on the pair $(r,\delta)$, it
 is reasonable to use the codimension in the space $[0,1] \times j^k \mathcal{O}(0)$, (where $j^k \mathcal{O}(0)$ denotes the finite dimensional space of $k$-jets at $0$), to define the genericity of the local solution.
  At the $k$-jet level, an $m$-parameter
 family of local solutions is given by a family:
 $\Phi(s,z)= (r(s),j^k \delta^s(z))$. For a cuspidal edge singularity, there is, generically, one condition ($r(s)=1/\sqrt{2}$) on $\Phi(s,-)$,
 and for both the swallowtail and cuspidal cross-cap there is one additional condition. In other words the space of cuspidal edges is a co-dimension $1$ submanifold whilst the other two are co-dimension $2$ submanifolds. A generic $2$-parameter family of solutions intersects these spaces transversely.  Therefore, at the surface level, each of these singularity types is generic.  Finally, these are the only singularities in a generic family of solutions: a solution at $s$ is singular if and only if $r(s) = 1/\sqrt{2}$. Given this, if both the real and imaginary parts of $\delta^s(0)$ are non-zero then the singularity
 is necessarily a cuspidal edge. Otherwise, we can have one more condition in the generic case, and this is either $\Im (\delta^s(0))=0$, or
 $\Re (\delta^s(0))=0$, so either a swallowtail or cuspidal cross cap.

\section{Singularities via the Cauchy problem}  \label{sec:cauchy}
\subsection{The Cauchy problem for a harmonic map into $\mathbb S^2$}
In view of Example \ref{singexample}, it seems interesting to compare the behaviour of the harmonic map $N$ along the equator with the type of singularity of the corresponding maximal surface. Such an investigation involves solving the Cauchy problem along a curve, which we could then take to be the equator.  We therefore describe next how to solve the Cauchy problem for a harmonic map into $\SSS^2$ via the DPW method.

\begin{theorem}  \label{cauchythm}
 Let $J$ be an interval, and suppose given analytic maps,
 $N_0: J \to \SSS^2$ and $W: J \to \SSS^2$, such that $\langle W,N_0 \rangle =0$. Then, on an open set $U \subset \C$ containing the real interval $J \times \{0 \}$, there is a unique harmonic map 
 $N: U \to \SSS^2$ satisfying the Cauchy conditions:
 \[
  N(x,0)=N_0(x), \quad \quad N_y(x,0) = W(x).
 \]
The solution is produced via the DPW method with potential:
\[
 \xi(z) = \frac{1}{2}\left(((\kappa_1(z)-i) E_1 + \kappa_2(z) E_2)\lambda^{-1}
 + 2\kappa_3(z) E_3 + (\kappa_1(z) + i) E_1 + \kappa_2(z) E_2) \lambda\right) \dd z,
\]
where $\kappa_1$, $\kappa_2$ and $\kappa_3$ are holomorphic extensions of the functions:
\beq  \label{abcformula}
 \kappa_1(x) = \langle N_0'(x), W (x) \rangle, \;
 \kappa_2(x) = \langle N_0'(x), N_0(x) \times W(x) \rangle, \;
 \kappa_3 = \langle  W'(x), N_0(x) \times W(x) \rangle.
\eeq
Conversely, any harmonic map into $\SSS^2$ can be locally represented this way in terms of Cauchy data along a curve.
\end{theorem}

\begin{proof}
The proof is a  modification of the proof of Theorem 4.4 of \cite{spherical},
replacing the Cauchy data for a spherical surface with the Cauchy data
for its harmonic Gauss map. Since a harmonic map into $\SSS^2$ corresponds exactly to a spherical surface, we only need to demonstrate 
the formulas for $\kappa_i$ above. An extended frame $F$ for the solution has Maurer-Cartan form: 
$F^{-1} \dd F = (U_p \lambda^{-1} + U_\mathfrak{k}) \dd z + (-\bar \Uk^t - \bar \Up^t \lambda ) \dd \bar z$. The Cauchy problem can be solved 
by finding the restriction of this to the real line and then extending holomorphically to 
obtain the DPW potential for the solution.
The solution potential is thus a holomorphic extension of:
\[
 F_0^{-1} \dd F_0 = (\Up \lambda^{-1} + \Uk-\bar \Uk^t -\bar \Up^t \lambda) \dd x,
\]
where $F_0$ is the extended frame for the harmonic map along the curve $y=0$. The frame used in \cite{spherical} satisfies:
\beq \label{framedef}
 N= \Ad_F E_3, \quad  N_y = -\Ad_F E_2, \quad N \times N_y = \Ad_F E_1,
\eeq
where $E_1$, $E_2$ and $E_3$ are an orthonormal basis for $\mathfrak{su}(2)$, chosen such that the Lie bracket is the same as the cross-product. Since $N$ has no branch points, we assume without loss of generality that $N_y \neq 0$.  
 Using $N_y = -\Ad_F E_2$ along $y=0$ we can write 
 \[
 \Up = (1/2)((\kappa_1-i)E_1 + \kappa_2 E_2), \quad \quad 
 \Uk = (1/2)(\kappa_3+di) E_3,
 \]
 where $\kappa_i$
 and $d$ are some real-valued functions. Then
 \[
  W'(x) = N_{yx}(x,0) = -\Ad_F([U-\bar U^t, E_2]),
  \quad \quad F^{-1}F_z = U-\bar U^t.
 \]
Since $\Up-\bar \Up^t = \kappa_1 E_1+\kappa_2 E_2$ and $\Uk-\bar \Uk^t= \kappa_3 E_3$, this becomes, using \eqref{framedef} along $y=0$,
\[
 W'(x) = -\kappa_1 N_0(x)+ \kappa_3 N_0 \times W(x).
\]
Similarly we obtain, along $y=0$:
\[
 N'(x) = \kappa_1  W(x) + \kappa_2 N_0(x) \times W(x).
\]
These give the claimed formulae for $\kappa_1$, $\kappa_2$ and $\kappa_3$. 
\end{proof}
%%%%%%%%%%%%%%%%
\begin{remark}
 For the purpose of numerical implementations, setting $E_1 = \odi(-i,-i)/2$,
$E_2 = \odi(1,-1)/2$ and $E_3 = \di(i,-i)/2$, the above formula for $\xi$ becomes:
\begin{align}\label{circlepot}
 \xi(z) =\frac{1}{4}\bbar 2 \kappa_3 i & \hspace{-1.5cm}(\kappa_2-1-\kappa_1 i)\lambda^{-1} + (\kappa_2 +1-\kappa_1 i)\lambda \\
 (-\kappa_2-1-\kappa_1 i)\lambda^{-1} +(1-\kappa_2-\kappa_1 i)\lambda & -2 \kappa_3 i \ebar \dd z.
 \end{align} 
 From this we see that the regularity condition
 $a(z) \neq 0$, which ensures that the harmonic map is
 nowhere holomorphic is:
 \beq \label{gcpReg}
\kappa_2- \kappa_1 i \neq 1.
 \eeq
\end{remark}

 %%%%%%%%%%%%%%%%%%%%%%%%%%%%%%%%%%%%%%
\subsection{The Cauchy problem along the equator}
Using Theorem \ref{cauchythm}, we can write down a potential for 
an arbitrary harmonic map $N: \C \supset \Omega \to \SSS^2$ that 
takes the real line into a great circle. For the case of the horizontal equator, all possible Cauchy data $N$ and $W$ are of the form:
\beq \label{cauchydata}
 N(x) = (\cos(\theta(x)), \sin(\theta(x)),0), \quad
 W(x) = -\cos(\phi(x)) V(x) + \sin(\phi(x))(0,0,1),
\eeq
where $V(x)=(-\sin(\theta(x)),\cos(\theta(x)),0)$ is the unit tangent 
to $N$, and $\theta(x)$, $\phi(x)$ are arbitrary real-valued functions such that the speed
\[
v(x):= \theta'(x)  \neq 0,
\]
for all $x$. The formulas
at \eqref{abcformula} are here:
\beq \label{abc2}
 \kappa_1(x) = -v(x) \cos(\phi(x)), \quad
 \kappa_2 (x) = -v(x) \sin(\phi(x)), \quad
 \kappa_3 (x) = -\phi'(x).
\eeq
The regularity condition \eqref{gcpReg} is
$iv e^{i\phi} \neq 1$, which becomes:
\[
 (v,\phi) \neq \pm (1, -\pi/2).
\]

If we take the potential \eqref{circlepot}, with these values of $\kappa_i$, and integrate with the identity as the initial condition, we obtain a harmonic map $\tilde N$ that maps the real line to a great circle in $\SSS^2$. Because of the choice of frame \eqref{framedef}, and the initial condition,
we can deduce that the circle is in the plane spanned by
$\tilde N(0,0) = E_3$ and $\tilde N_x(0,0)$.
To obtain a map that maps the real line to the equator in the $E_1 E_2$-plane, we need to choose a different initial condition, namely an element $F_0 \in SU(2)$ that rotates rotates this plane to the $E_1 E_2$-plane. 
Differentiating the formula $\tilde N = \Ad_F E_3$, we have
\beqas
\tilde N_x (0,0) &= &\Ad_F [U - \bar U^t, E_3] \\
&=& [\kappa_1 E_1 + \kappa_2 E_2, E_3 ], \quad
\quad \hbox{at } z=0,\\
&=& v(0) ( \cos(\phi(0)) E_2 - \sin(\phi(0)) E_1).
\eeqas
Solving
\[ \Ad_{R_0} E_3 = E_1, \quad \quad 
 \Ad_{R_0} \tilde N_x(0,0) \propto  E_2\]
 gives the initial condition:
\beq \label{IC2}
R_0 = \frac{1}{\sqrt{2}} \bbar e^{-i\alpha}  & \lambda e^{i\alpha}\\ -\lambda^{-1} e^{-i \alpha} & e^{i \alpha} \ebar, \quad \quad \alpha=\frac{\phi(0)}{2}.
\eeq
In summary:
\begin{theorem}\label{singcauchythm}
 Let $v$ and $\phi$ be a pair of real analytic functions
 from an interval $I$ into $\real$, with $v(t)>0$.
 Let $N$ be the harmonic map in $\SSS^2$ given by the potential $\xi$ at \eqref{circlepot}, with $\kappa_i$ given by the holomorphic extensions of the functions at \eqref{abc2}. Then $N$ maps the real line into a great circle.  If the potential is integrated with 
 the initial condition \eqref{IC2}, then this great circle lies in the plane perpendicular to $E_3$.
\end{theorem}
%%%%%%%%%%%%%%%%%%%%%%%%%%%%%%

\subsection{Singularity type via the Cauchy problem}
 As always, the product of the $(1,1)$- and $(2,1)$- entries
 of the initial condition $R_0$ in \eqref{IC2} with $\lambda =1$
 gives $(i g \omega)(p)$, so we have
\[
 (g\omega)(0) = i e^{-i \phi(0)}.
\]
To determine the singularity type using Corollary \ref{cor:singnormlDPW},
we first remove the diagonal term in \eqref{circlepot}
by the diagonal gauge  $D_+ = \di (\exp(-\int i \kappa_3/2\,  \dd z),
\exp(\int i\kappa_3/2\, \dd z))$ as $C_2 = C D_+$. Then $C_2$ 
 satisfies $\dd C_2 = C_2 \xi_2$ as 
\[
\xi_2(z)  = 
\begin{pmatrix}
0 & a(z) \lambda^{-1}  + c(z) \lambda \\   
b(z)\lambda^{-1} + d(z) \lambda & 0
\end{pmatrix} \dd z,
\]
where 
\begin{align*}
a &= \frac14 \exp\left(i \int_0^z \kappa_3\,  \dd z\right)
(\kappa_2-1-\kappa_1 i), 
\quad c = \frac14 \exp\left(i \int_0^z \kappa_3\,  \dd z\right)
(\kappa_2+1-\kappa_1 i),  \\ 
 b &= \frac14 \exp\left(-i \int_0^z \kappa_3\,  \dd z\right)
(-\kappa_2-1-\kappa_1 i), \quad 
d = \frac14 \exp\left(-i \int_0^z \kappa_3\,  \dd z\right)
(1-\kappa_2-\kappa_1 i).
\end{align*}
Scaling the coordinate $\int_0^z a(w) \dd w$ by
$-i e^{ i \phi (0)}$ in the proof of
Corollary \ref{cor:singnormlDPW}, gives the normalization $(g \omega) (0) =1$, and hence:
%%%%%%%%%%%%%%%%%%%%%%%%%%%
\begin{theorem}\label{singcauchythm2}
 Retain the assumptions in Theorem \ref{singcauchythm} and
 take the spacelike maximal surface $f$ in $\Nili$
 through the DPW method by the potential $\xi$ with the initial condition
 in \eqref{IC2}. Then $f$ is singular along the interval $I \subset \mathbb R$ containing $0$.
 Consider the conformal change of coordinates $\zeta=\zeta(z)$ given by:
 \[
\zeta(z) = i e^{-i\phi(0)}\int_0^z a(w) \dd w,
 \]
where $a(z)$ is the holomorphic extension of:
\[
 a(x) =\frac14
 \left(- v(x)\sin \phi(x) -1+i v(x) \cos \phi(x)\right)
  \exp(-i (\phi(x)-\phi(0))).
\]
Then the singularity criteria in
  Theorem \ref{thm:criteria2} can be applied to the coefficient
  function $\hat B(\zeta)$ of the Abresch-Rosenberg differential $\hat B(\zeta) \dd \zeta^2$. As a function of $z$, we have:
  \beq \label{Bhatformula}
  \hat B( \zeta(z)) = 
  \frac{i v(z)  -e^{i \phi(z)}}{i v(z) - e^{-i \phi(z)}}.
  \eeq
 \end{theorem}
 %%%%%%%%%%%%%%
 Note: to obtain the formula \eqref{Bhatformula}, write
 \[
  \hat B(\zeta) (\dd \zeta)^2 = -a(z) b(z) (\dd z)^2
  = \frac{b(z)}{a(z)} e^{2i\phi(0)} (\dd \zeta)^2,
 \]
and substitute 
\[
4a(z) = (\kappa_2(z)-1-\kappa_1(z)i)
e^{i \int_0^z \kappa_3 \dd z}, \quad
4 b(z) = (-\kappa_2(z)-1-\kappa_1(z) i) e^{-i \int_0^z \kappa_3 \dd z},
\]
 using the formulas
 in terms of $v$ and $\phi$ at \eqref{abc2}.
 %%%%%%%%%%%%%%%%%%%
\begin{remark}
At the point $z=\zeta=0$, the values that determine
the singularity type are:
\[
1-\Re(\hat B)\big|_{z=0} =
\frac{2 (v + \sin \phi) \sin \phi }
{1+2v \sin \phi + v^2},
\]
\[
 \Im(\hat B)\big|_{z=0} =
  \frac{2 \sin \phi \cos \phi}
  {1+ 2v \sin \phi + v^2}.
\]
According to Theorem \ref{thm:criteria2}, we have
a cuspidal edge at $z=0$ provided:
\beq \tag{CE} \label{ce}
0 \neq \phi(0) \neq \pm \pi/2 \quad \hbox{and} \quad
v(0) \neq -\sin (\phi(0)).
\eeq
Consider now the criteria of Theorem \ref{thm:criteria2} for cuspidal cross-cap and swallowtail singularities.
The condition $\Im(\hat B)' \neq 0$, needed for both singularities, will hold generally provided $\phi'(0) \neq 0$. Assuming this, arbitrary choices of $v$ and $\phi$ that make the first expression above vanish but not the second, i.e.:
\beq \tag{ST} \label{st}
 v(0)=- \sin (\phi(0)), \quad \quad 0 \neq \phi(0) \neq \pm \pi/2
\eeq
will have a swallowtail singularity at $z=0$. Likewise, choices that make the second expression vanish but not the first, i.e.:
\beq \tag{CC} \label{cc}
 \phi(0) = \pm \pi/2, \quad  \quad v(0) \neq -\sin (\phi(0)),
\eeq
will result in a cuspidal cross-cap singularity.

From \eqref{cauchydata}, $\phi(0)$ is the angle between $N_y$ and the $E_1 E_2$-plane. Thus,
geometrically, the behaviour of the Gauss map at a cuspidal cross-cap is clear: if coordinates are chosen
so that $N_x$ is parallel to the equator, then the cuspidal cross-cap occurs generically if $N_y(0)$ is
parallel to the $E_3$ axis.  On the other hand, for a swallowtail, the angle between $N_y(0)$ and the $E_1 E_2$-plane must be a specific value depending on the ratio $v(0)/1$ of the speed of $N_x$ to the (unit) speed of $N_y$, namely $\phi(0) = \arcsin(-v(0))$.
\end{remark}
 \begin{figure}[htb]
\centering
$
\begin{array}{ccc}
\includegraphics[height=27mm]{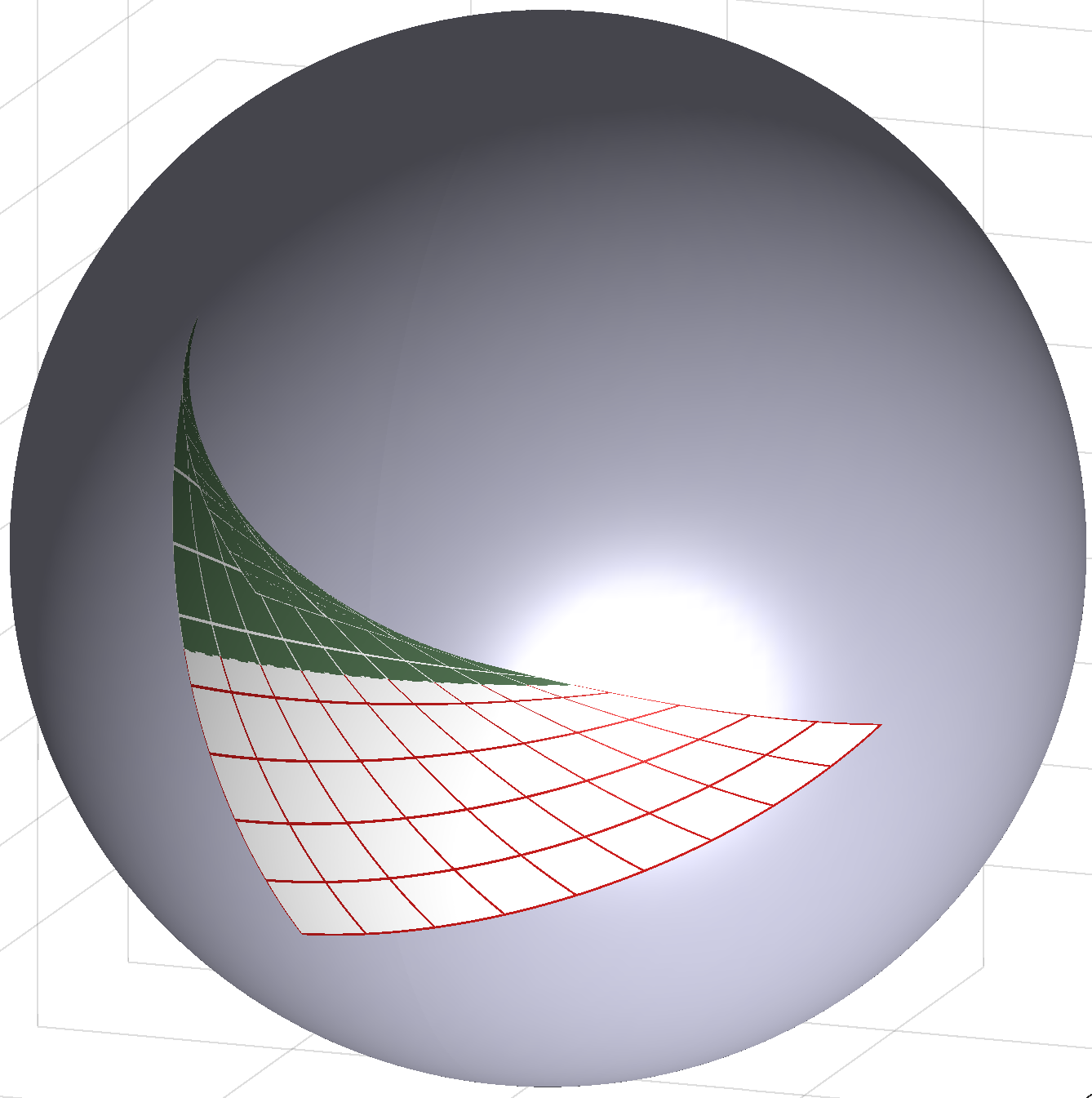} \quad
 & \quad
\includegraphics[height=27mm]{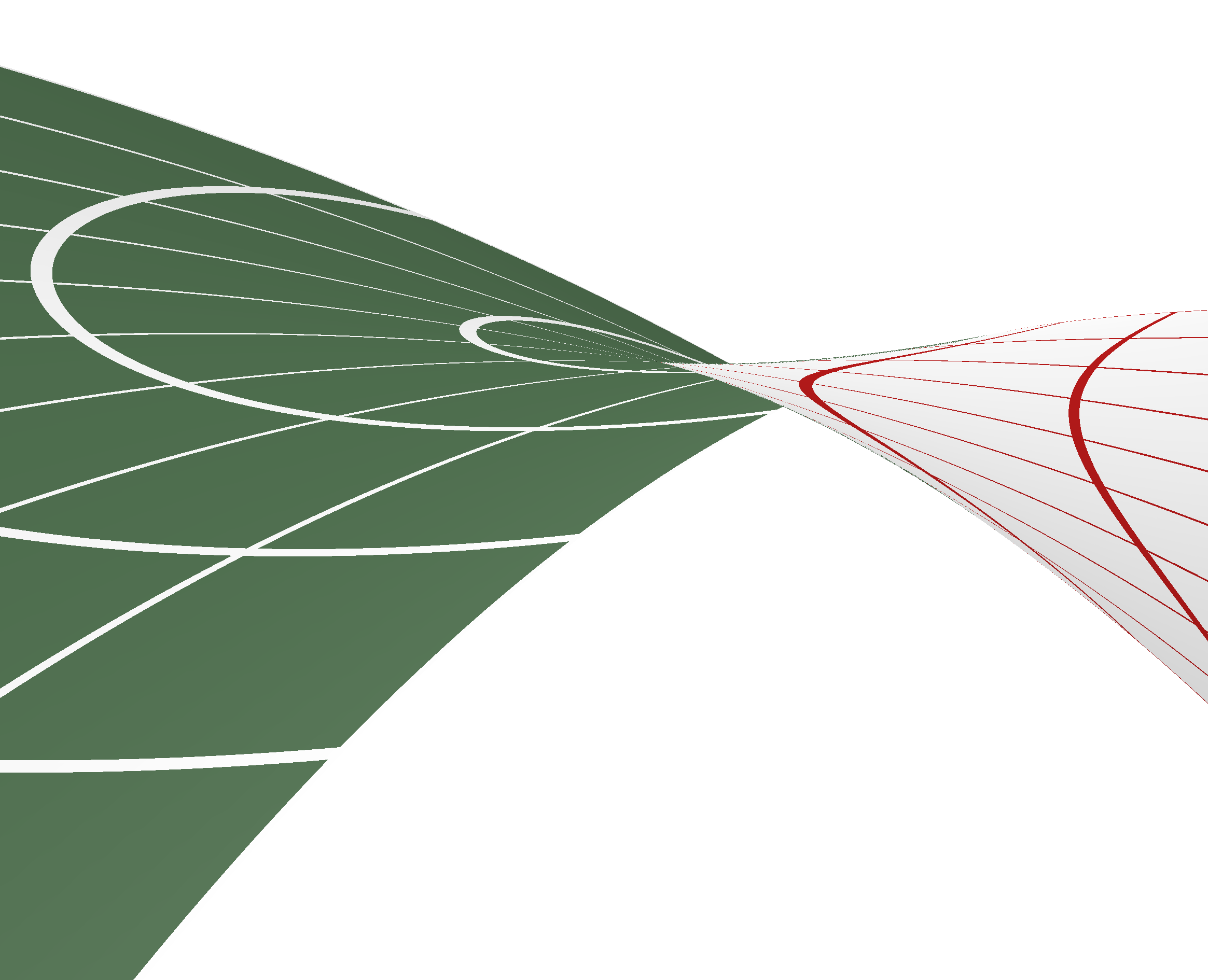} \quad & \quad
\includegraphics[height=27mm]{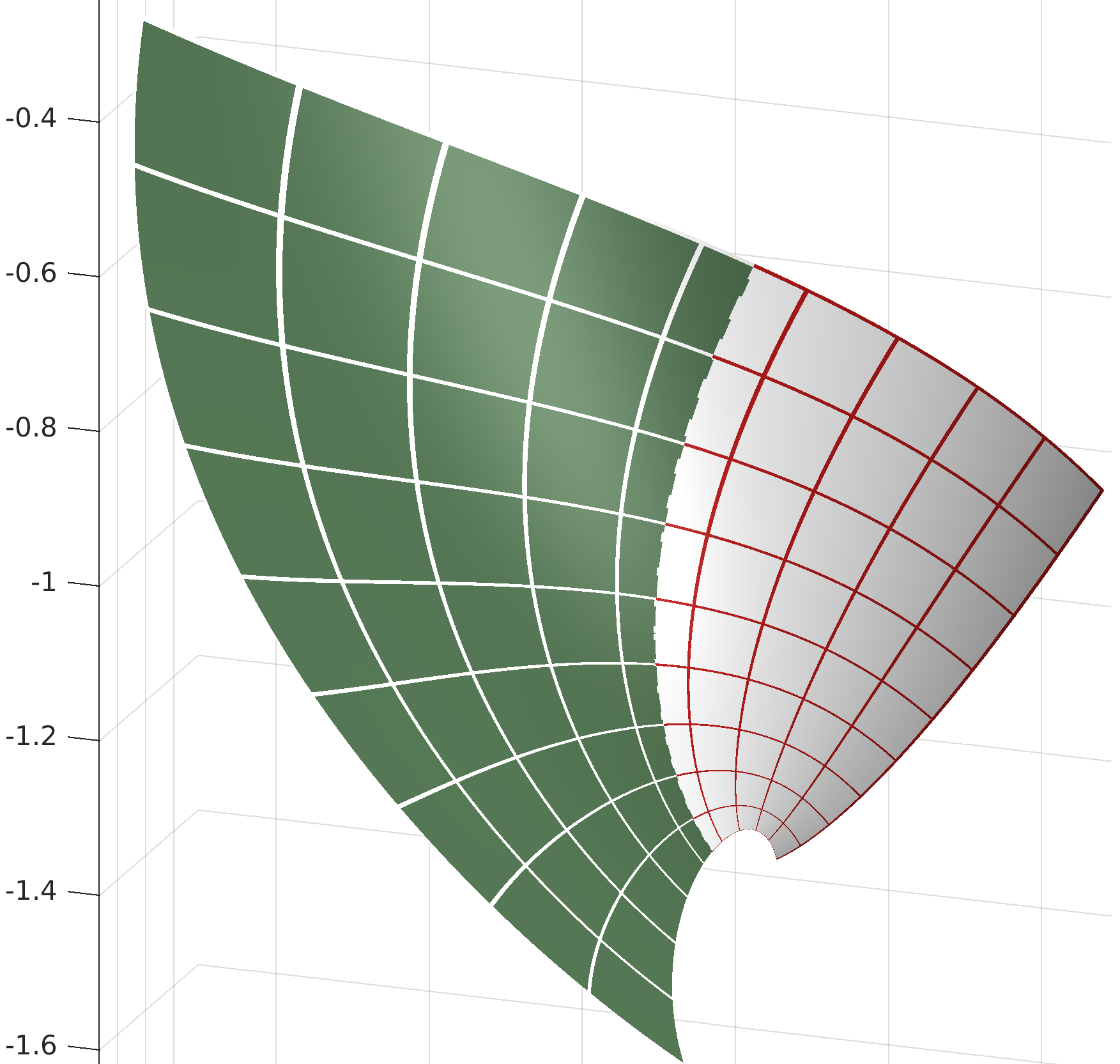}
\vspace{1ex} \\
\includegraphics[height=27mm]{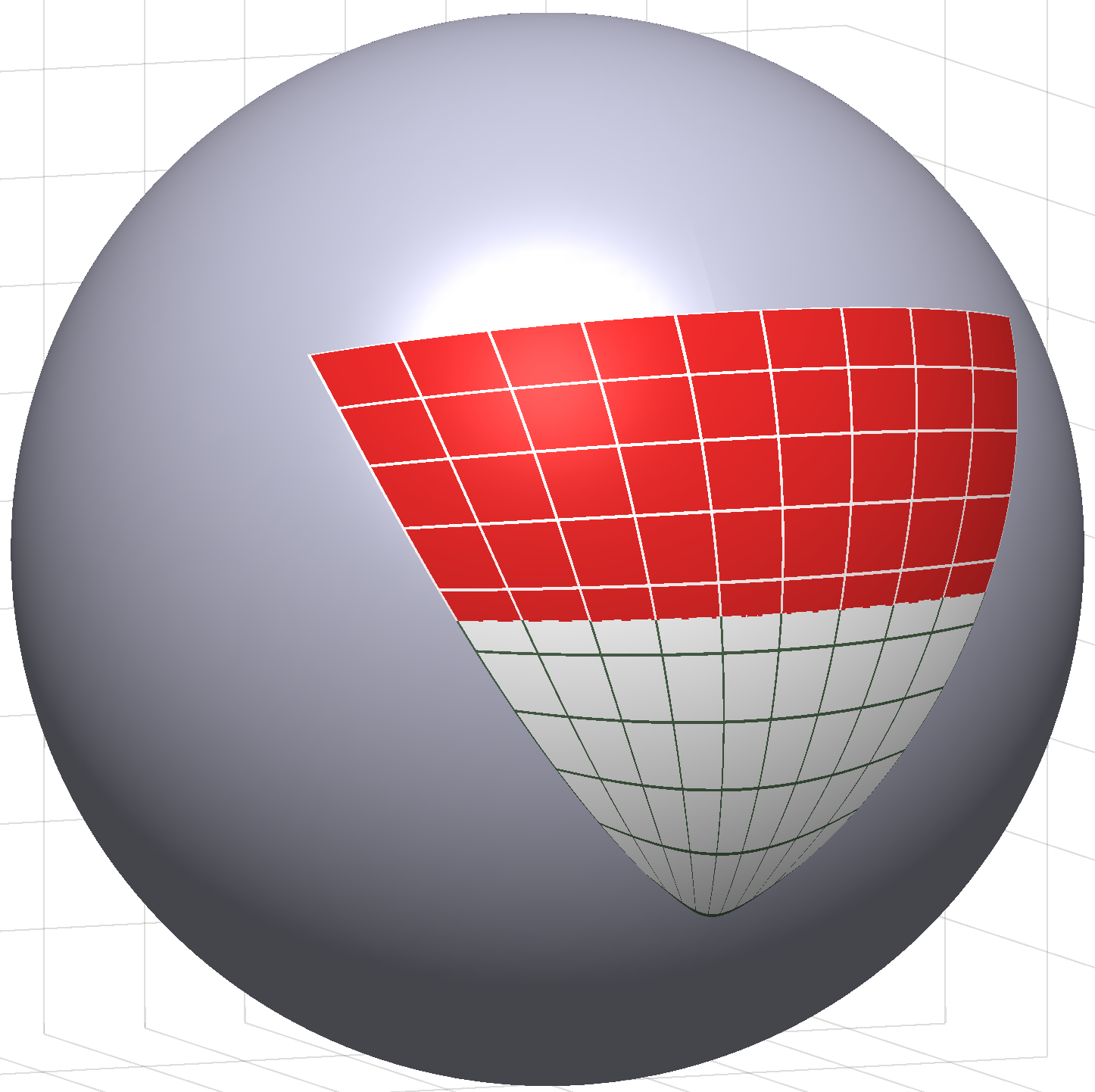}  \quad & \quad
\includegraphics[height=27mm]{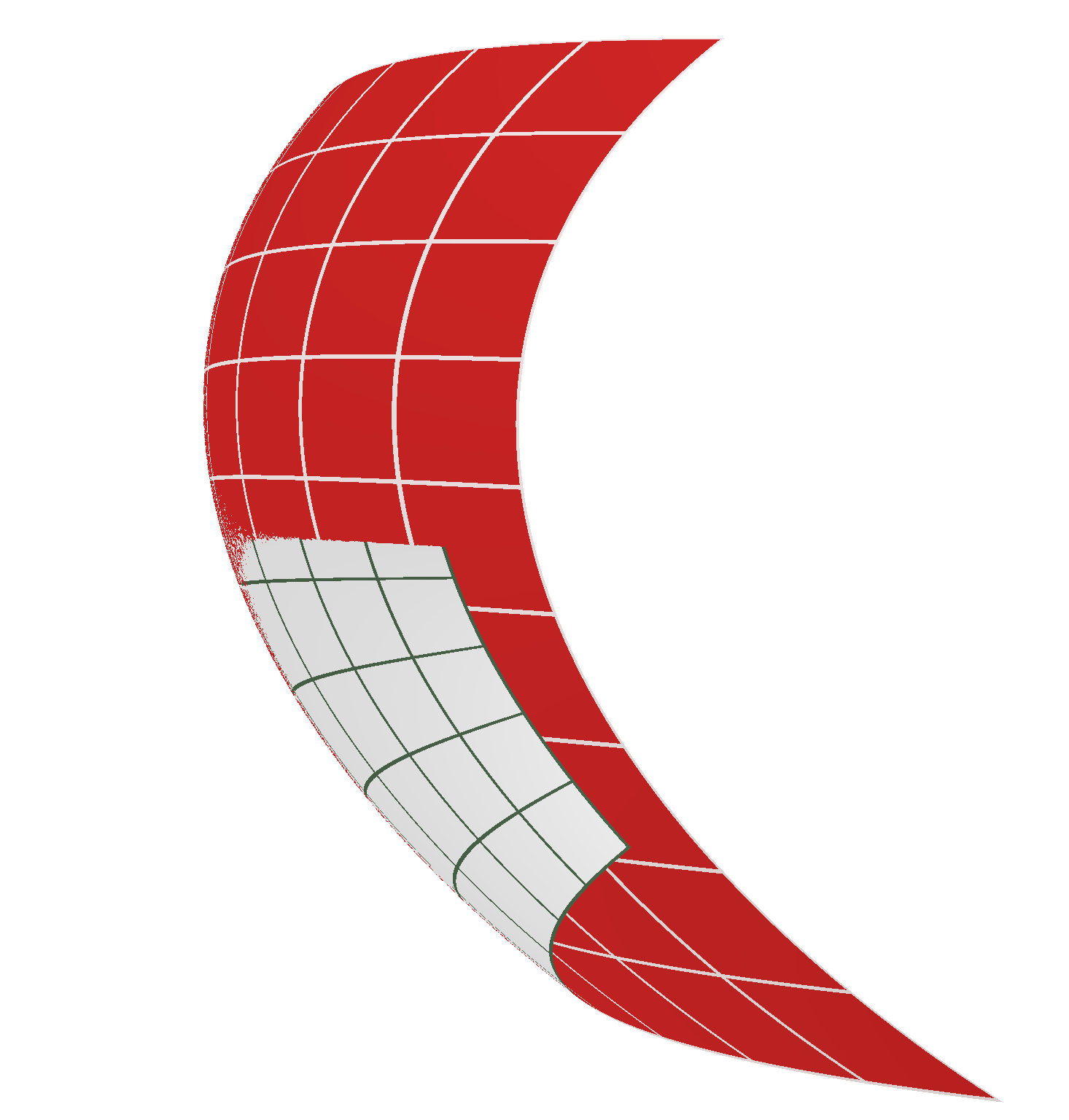}   \quad & \quad
\includegraphics[height=27mm]{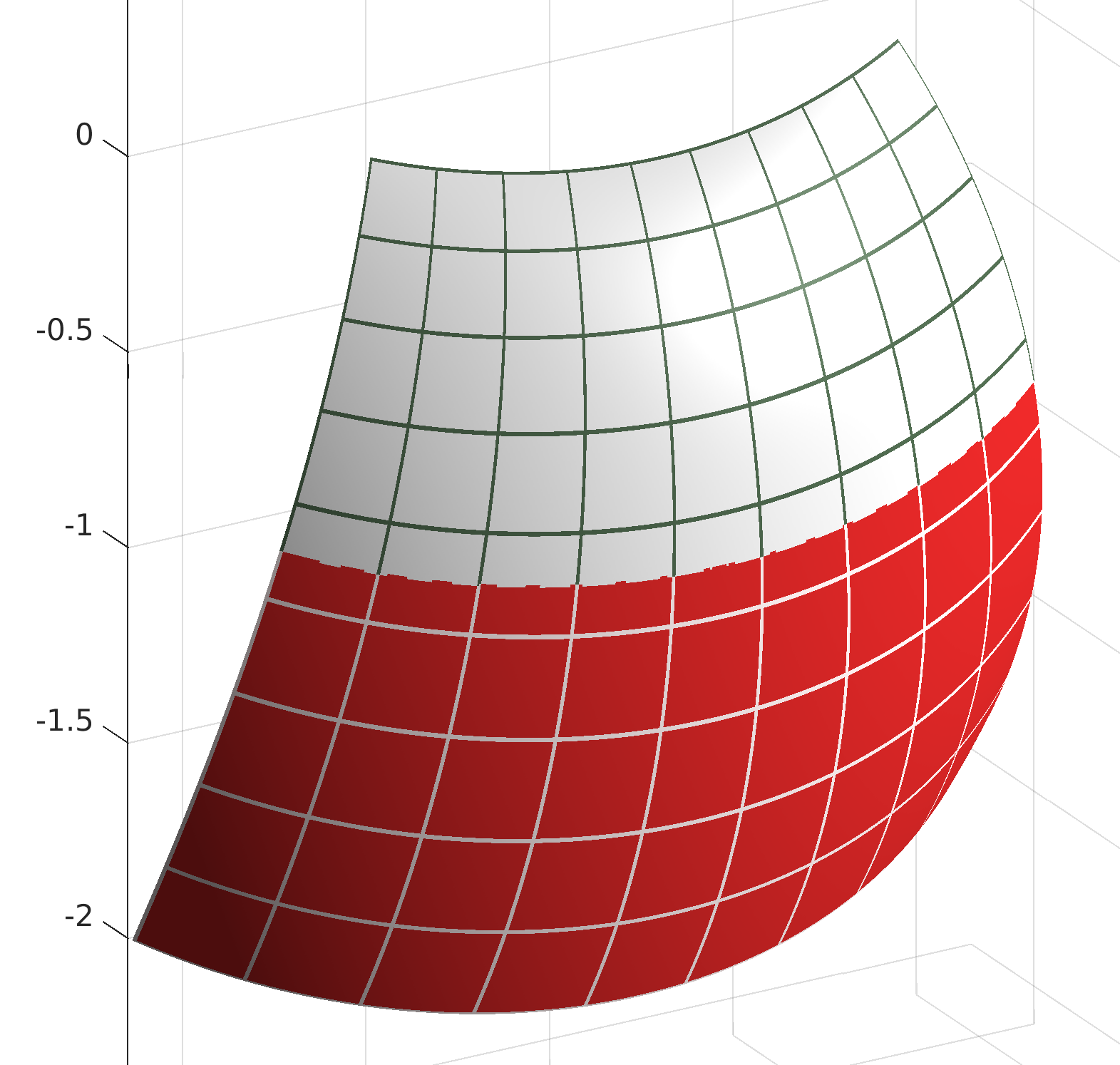}
\end{array}
$
\caption{Solutions to the geometric Cauchy problem using Theorem \ref{singcauchythm2}.
 See Example \ref{examplegcp}. Top: swallowtail. Bottom: cuspidal cross-cap.}
\label{figgcp}
\end{figure}

\begin{example} \label{examplegcp}
A local solution for a swallowtail, taking $v(x)=-1/\sqrt{2}$ and $\phi(x) = x+\pi/4$, and for a cuspidal cross-cap, taking $v(x)=1$ and $\phi(x)=x + \pi/2$
are computed and shown in Figure \ref{figgcp}.
 Left to right are shown the Gauss map, the maximal surface in $\Nili$, and the CMC surface in $\real^3$ with the same Gauss map.
\end{example}

%%%%%%%%%%%%%%%%%%%%%%%%%%%%%%

\section{Geometry of the associated CMC surface and an application}  \label{sect:equatorial}

\subsection{The equatorial curve on the CMC surface}
Let $N: \D \to \SSS^2$ be a nowhere holomorphic harmonic map,
fix an embedding of $\SSS^2$ into $\real^3= \mathfrak{su}(2)$,
and let $f$ and $f_{cmc}$ denote the associated maximal suface in $\Nili$,
and CMC $1/2$ surface in $\real^3$ respectively (see Theorem \ref{thm:Sym}).  As in Example \ref{singexample}, let $C_{cmc}$ denote the set
$f_{cmc}(C)$ where $C \subset \D$ is the preimage under $N$ of the equator in the $E_1 E_2$-plane, which we may call the
\emph{equatorial curve} in the image of $f_{cmc}$. It is the set of
points on the CMC surface where the surface normal is perpendicular to $E_3$. The map $f_{cmc}$ is regular, hence $C_{cmc}$ is a regular curve in $\real^3$.
\begin{theorem}  \label{thm:equatorial}
Let $N$, $f$ and $f_{cmc}$ be above, and let $p$ be a
non-degenerate singular point.
 The maximal surface $f$ has a
 cuspidal cross-cap singularity at $p$ if and only if
 the tangent at $f_{cmc}(p)$ to the equatorial curve is
 perpendicular to $E_3$,
 and a swallowtail at $p$ if and only if the
 tangent is parallel to $E_3$.
\end{theorem}
\begin{proof}
 This follows from the equations \eqref{eq:cmc} relating a CMC surface to its harmonic Gauss map. We can always locally choose coordinates as in Theorem \ref{singcauchythm}, with $z=0$ corresponding to $p$.
 Then the condition \eqref{cc}, namely $\phi(0)= \pm \pi/2$ and
 $v(0) \neq -\sin(\phi(0))$ is equivalent to $f$ having a cuspidal
 cross-cap singularity. In the notation of Section \ref{harmonicintro},
 we have  $f_{cmc}= f^-$, and from \eqref{eq:cmc} this satisfies the equations:
 \[
 f^-_x = N \times N_y \,  -N_x,
 \quad \quad
 f^-_y = -N \times N_x \,  -N_y.
\]
Since $\phi(0)=\pm \pi/2$, we have $N_x \perp N_y$ at $z=0$, hence:
\[
 f^-_x(0) \parallel N_x(0), \quad  \hbox{and} \quad f^-_y(0) \parallel N_y(0).
\]
 Since $\gamma(x)= f^-(x,0)$ is locally
 the equatorial curve, the tangent $f^-_x(0)$ to the curve at $p$ is parallel to $N_x(0)$, which is
 perpendicular to $E_3$.  Conversely, if the tangent is perpendicular to $E_3$ then, by the same equations, $N_y$ must be perpendicular to $N_x$, so $\phi(0)=\pm \pi/2$. The other condition, $v(0) \neq -\sin(\phi(0))$, holds because we assumed that the singular set is non-degenerate.

 For the swallowtail condition, again using the setup  \eqref{cauchydata} of  Theorem \ref{singcauchythm}, we have:
 \[
 N_x(0) = v(0)V(0), \quad \quad N_y(0) = -\cos(\phi(0)) V(0) + \sin (\phi(0)) E_3,
 \]
 and $N(0) \times V(0) = E_3$.  This give:
 \beqas
 f^-_y(0) &=& - N \times N_x - N_y \\
 &=&  -(v(0)+\sin(\phi(0))) E_3 + \cos (\phi(0)) V(0),
 \eeqas
 so $v(0) = -\sin(\phi(0))$ if and only if 
 $f^-_y(0)$ is parallel to $V(0)$, or
 equivalently, since $f^-$ is conformally immersed, $f^-_x(0)$ is parallel to $f_y \times N$, 
 i.e. parallel to $V(0) \times N(0) = E_3$. 
 For the converse direction, if the equatorial curve is parallel to $E_3$, the second condition for the swallowtail, $0 \neq \phi(0) \neq \pm \pi/2$, holds because $f^-$ is regular, so $f^-_y$ above cannot be zero.
\end{proof}
\subsection{Spacelike maximal discs with natural boundary} \label{sect:disc}
As shown in Section \ref{generalizedsection}, there are no complete spacelike maximal surfaces in
$\Nili$.
The simplest ``global'' regular surface one can contemplate can be defined as follows:
\begin{definition}
Let $\SSS^2_+ := \{ x \in \SSS^2 ~|~ x_3 > 0\}$, and denote its closure by $\overline{\SSS^2_+}$.
Let $\D$ be a contractible domain in $\C$ and $N: \D \to \SSS^2$ be a nowhere holomorphic harmonic map, with the property that $N$ maps a closed set contained in $\D$ diffeomorphically
to  $\overline{\SSS^2_+}$. Denote
by $\Omega$ the open subset $N^{-1}(\SSS^2_+)$. Then the corresponding generalized maximal surface
$f: \Omega \cup \partial \Omega \to \Nili$ is called a \emph{(spacelike) maximal disc with null boundary}.
\end{definition}
Note that the restriction $f|_\Omega: \Omega \to \Nili$ is a regular maximal surface, which extends to a generalized maximal surface $f: \D \to \Nili$, and that the boundary $\partial \Omega$ is contained
in the singular set of $f$.
%%%%%%%%%%%%%%%%%%%%%%%%%
\begin{example} \label{ex:disc}
 We can construct examples of a maximal disc with null boundary by taking the DPW potential for  a CMC sphere, of the form $\xi = {\tiny{\bbar 0 & 1 \\ 0 & 0 \ebar}} \lambda^{-1} \dd z$ and deforming it
 slightly:
\[ \xi(z) = \bbar 0 & 1 \\ -\epsilon(z) & 0 \ebar \lambda^{-1} \dd z, \]
where $\epsilon(z)$ is some small holomorphic function, for instance a small constant. An example with $\epsilon(z) = 0.04$ is computed and shown
in Figure \ref{figdisc}.
\end{example}
 \begin{figure}[htb]
\centering
$
\begin{array}{ccc}
\includegraphics[height=30mm]{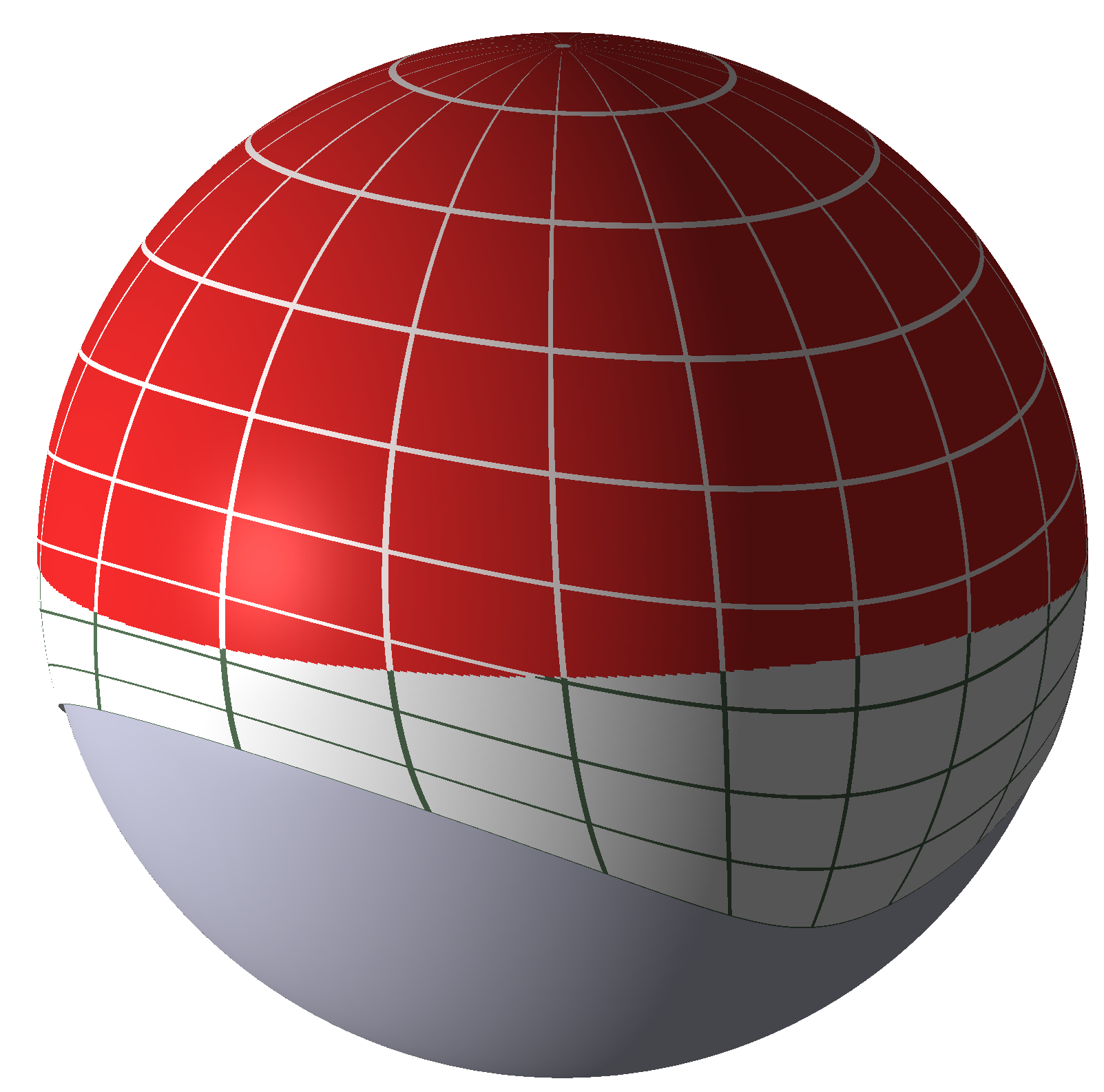}
 & \quad
\includegraphics[height=28mm]{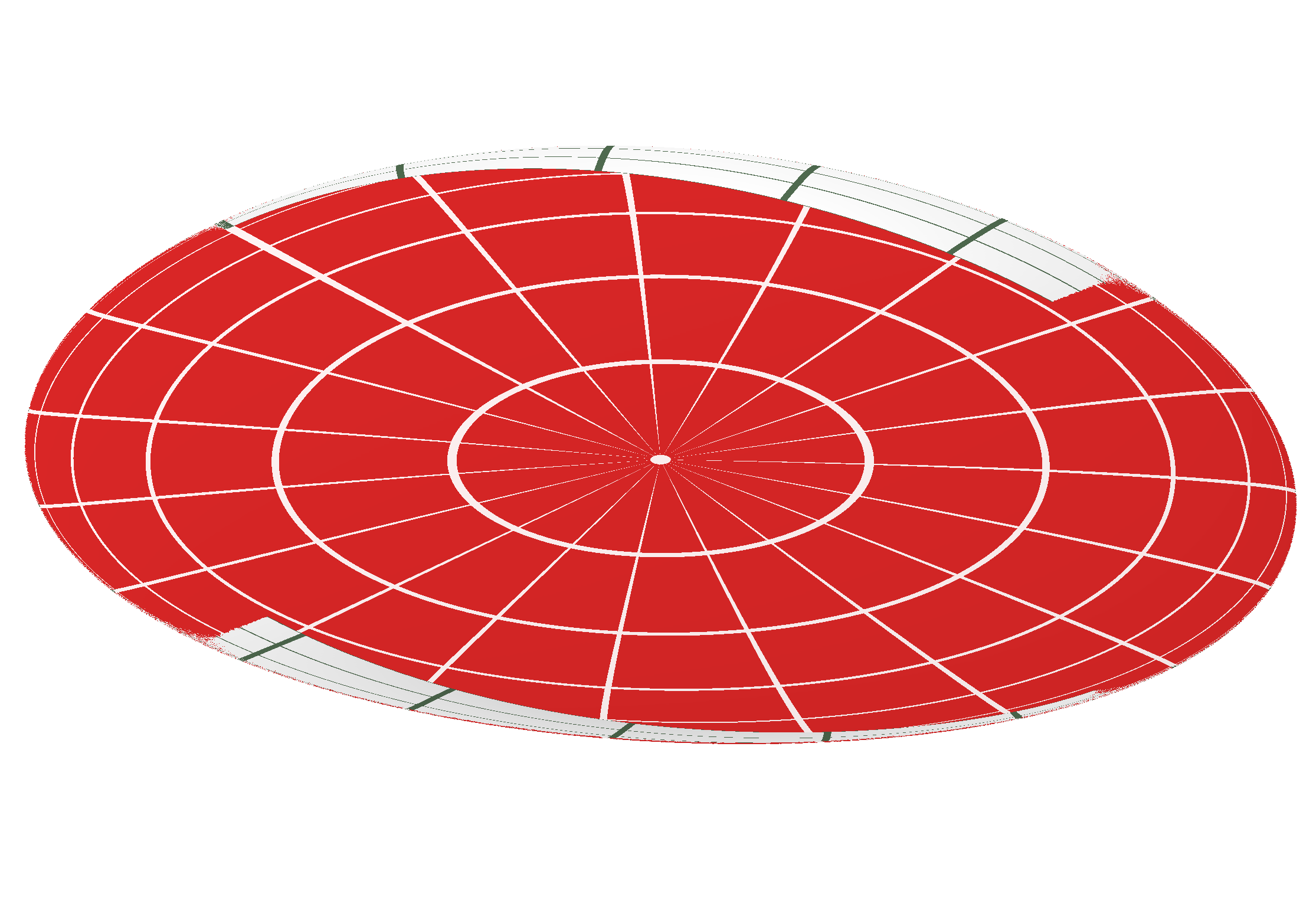} & \quad
\includegraphics[height=30mm]{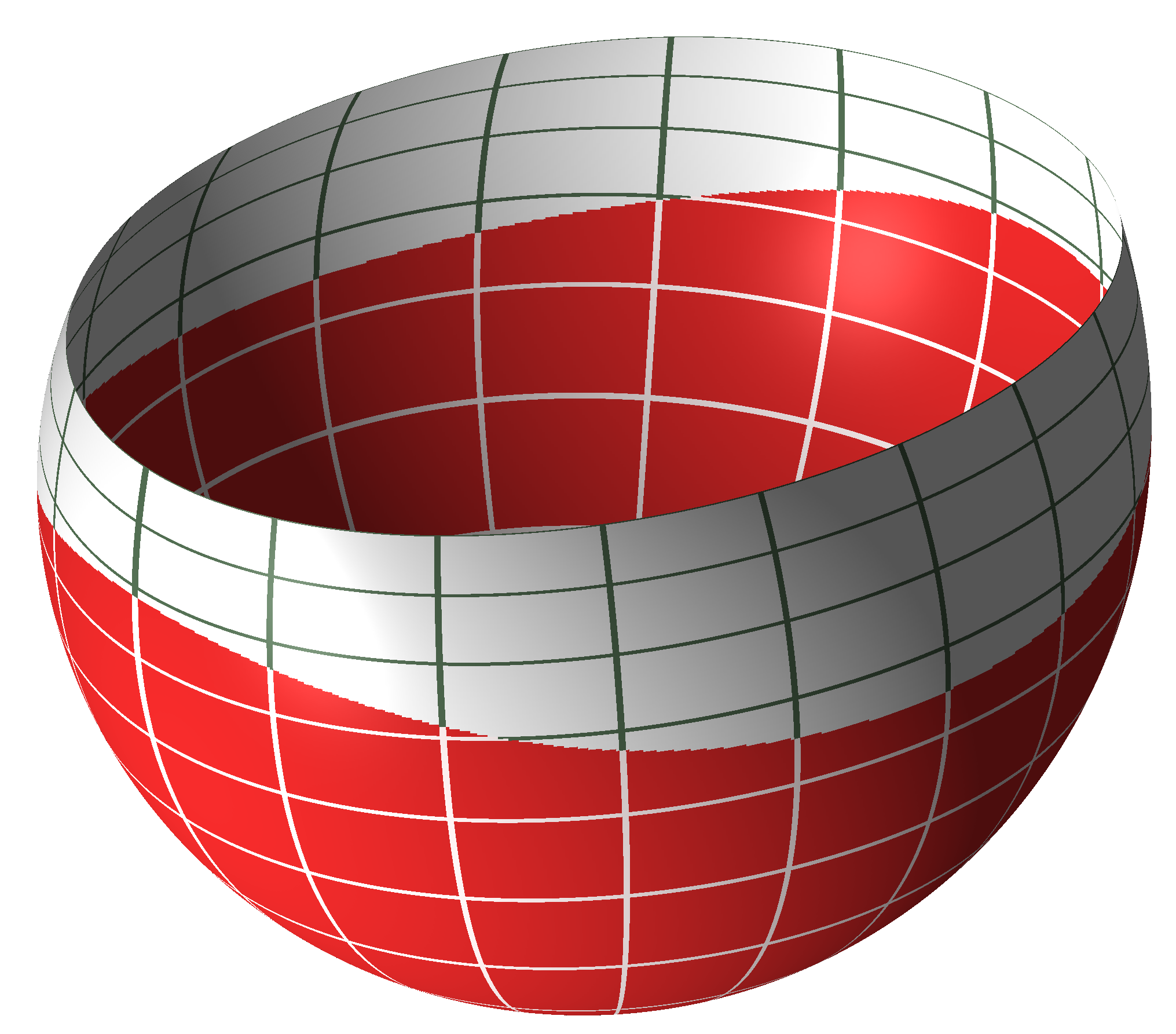}
\end{array}
$
\caption{Example \ref{ex:disc}. Left: Gauss Map. Middle: Maximal surface. Right: CMC surface.}
\label{figdisc}
\end{figure}
The example shown in Figure \ref{figdisc} has four cuspidal cross-cap singularities on the singular curve. These correspond to the four points on the equatorial curve of the
corresponding CMC surface (right) where the tangent is horizontal, as described in Theorem \ref{thm:equatorial}.
This illustrates a general principle:
\begin{theorem}  \label{thm:twocrosscap}
 Let $f: \Omega \cup \partial \Omega \to \Nili$ be a spacelike maximal disc with null boundary. Suppose that the boundary singular curve is non-degenerate.
 Then the boundary curve has at least two cuspidal cross-cap singularities.
\end{theorem}
\begin{proof}
 The corresponding CMC surface in $\real^3$ is regular, since we assumed the harmonic Gauss map is nowhere holomorphic. The non-degeneracy assumption implies that the boundary curve $C= \partial \Omega$ is a regular curve in the open set $\D$ containing $\Omega \cup \partial \Omega$. Since $C$ is the inverse image under the Gauss map $N$ of the equator, it is a closed curve, and hence $f_{cmc}(C)$ is a regular closed curve in $\real^3$. The third component function of the curve
 $\gamma: C \to \real^3$ given by $f_{cmc}|_C$ must have at least one local maximum and at least one local minimum. At these points the tangent is parallel to the $E_1E_2$-plane. Hence the claim follows from  Theorem \ref{thm:equatorial}.
\end{proof}
\begin{remark} Theorem \ref{thm:twocrosscap} rules out the existence of a spacelike maximal disc with a cuspidal edge as boundary. Note that for the case where the CMC surface is a sphere the equatorial curve is everywhere parallel to the $E_1 E_2$-plane. In this case the singular curve is everywhere degenerate (since $B(z)=0$). The corresponding maximal surface is a flat disc with a fold singularity at the boundary.
\end{remark}

\bibliographystyle{plain}
 \bibliography{mybib}

\begin{thebibliography}{10}

\bibitem{BerTai2005}
D~Berdinskii and I~Taimanov.
\newblock Surfaces in three-dimensional {L}ie groups.
\newblock {\em Siberian Math. J.}, 46:1005--1019, 2005.

\bibitem{bobenko1994}
A~I Bobenko.
\newblock Surfaces in terms of 2 by 2 matrices. {O}ld and new integrable cases.
\newblock In {\em Harmonic maps and integrable systems}, number E23 in Aspects
  Math., pages 83--127. Vieweg, 1994.

\bibitem{spherical}
D~Brander.
\newblock Spherical surfaces.
\newblock {\em Exp. Math.}, 25(3):257--272, 2016.

\bibitem{mincmc}
D~Brander and JF~Dorfmeister.
\newblock Deformations of constant mean curvature surfaces preserving
  symmetries and the {H}opf differential.
\newblock {\em Ann. Sc. Norm. Super. Pisa Cl. Sci. (5)}, XIV:1--31, 2015.

\bibitem{cartier2011}
S~Cartier.
\newblock {\em Surfaces des espaces homog\`enes de dimension 3}.
\newblock PhD Thesis. Universit\'e Paris-Est, 2011.

\bibitem{cmo2017}
A~Cintra, F~Mercuri, and I~Onnis.
\newblock Minimal surfaces in {L}orentzian {H}eisenberg group and
  {D}amek–{R}icci spaces via the {W}eierstrass representation.
\newblock {\em J. Geom. Physics}, 121:396--412, 2017.

\bibitem{daniel2011}
B~Daniel.
\newblock The {G}auss map of minimal surfaces in the {H}eisenberg group.
\newblock {\em Int. Math. Res. Not.}, pages 674--695, 2011.

\bibitem{dik2016}
J~Dorfmeister, J~Inoguchi, and S-P Kobayashi.
\newblock A loop group method for minimal surfaces in the three-dimensional
  {H}eisenberg group.
\newblock {\em Asian J. Math.}, 20:409--448, 2016.

\bibitem{DorPW}
J~Dorfmeister, F~Pedit, and H~Wu.
\newblock Weierstrass type representation of harmonic maps into symmetric
  spaces.
\newblock {\em Comm. Anal. Geom.}, 6:633--668, 1998.

\bibitem{figueroa1999}
C~Figueroa.
\newblock Weierstrass formula for minimal surfaces in {H}eisenberg group.
\newblock {\em Pro Mathematica}, 13:71--85, 1999.

\bibitem{figueroa2007}
C~Figueroa.
\newblock On the {G}auss map of a minimal surface in the {H}eisenberg group.
\newblock {\em Matematica Contemporanea}, 33:139--156, 2007.

\bibitem{fsuy}
S~Fujimori, K~Saji, M~Umehara, and K~Yamada.
\newblock Singularities of maximal surfaces.
\newblock {\em Math. Z.}, 259:827--848, 2008.

\bibitem{hos1982}
D~A Hoffman, R~Osserman, and R~Schoen.
\newblock On the {G}auss map of complete surfaces of constant mean curvature in
  {${\bf R}^3$} and {${\bf R}^4$}.
\newblock {\em Comment. Math. Helvetici}, 57:519--531, 1982.

\bibitem{Kenmotsu1979}
K~Kenmotsu.
\newblock Weierstrass formula for surfaces of prescribed mean curvature.
\newblock {\em Math. Ann.}, 245(2):89--99, 1979.

\bibitem{kiko2022}
H~Kiyohara and S~Kobayashi.
\newblock Timelike minimal surfaces in the three-dimensional {H}eisenberg
  group.
\newblock {\em J. Geom. Anal.}, 32(8):Paper No. 225, 2022.

\bibitem{krsuy}
M~Kokubu, W~Rossman, K~Saji, M~Umehara, and K~Yamada.
\newblock Singularities of flat fronts in hyperbolic space.
\newblock {\em Pacific J. Math.}, 221:303--351, 2005.

\bibitem{hlee2011}
H~Lee.
\newblock Maximal surfaces in {L}orentzian {H}eisenberg space.
\newblock {\em Differential Geom. Appl.}, 29:73--84, 2011.

\bibitem{lmm2011}
JH~Lira, M~Melo, and F~Mercuri.
\newblock A {W}eierstrass representation for minimal surfaces in 3-dimensional
  manifolds.
\newblock {\em Results. Math.}, 60:311--323, 2011.

\bibitem{milnor1976}
J~Milnor.
\newblock Curvatures of left invariant metrics on {L}ie groups.
\newblock {\em Adv. Math.}, 21:293--329, 1976.

\bibitem{rahmani2006}
N~Rahmani and S~Rahmani.
\newblock Lorentzian geometry of the {H}eisenberg group.
\newblock {\em Geom. Dedicata}, 118:133--140, 2006.

\bibitem{rahmani1992}
S~Rahmani.
\newblock M\'etriques de {L}orentz sur les groupes de {L}ie unimodulaires, de
  dimension trois.
\newblock {\em J. Geom. Phys.}, 9:295--302, 1992.

\bibitem{uy2006}
M~Umehara and K~Yamada.
\newblock Maximal surfaces with singularities in minkowski space.
\newblock {\em Hokkaido Math. J.}, 35:13--40, 2006.

\bibitem{Wu1999}
Hongyou Wu.
\newblock A simple way for determining the entials for harmonic maps.
\newblock {\em Ann. Global Anal. Geom.}, 17(2):189--199, 1999.

\end{thebibliography}

 \end{document}